\documentclass[11pt]{amsart}
    
\usepackage{amssymb}
\usepackage{lscape}

\usepackage{mathpazo}

\usepackage{mathrsfs}

\usepackage[all]{xy}
\usepackage[headings]{fullpage}
\usepackage{paralist}

\usepackage{tabularx,ragged2e,booktabs,caption}
\usepackage[colorlinks=true,linkcolor=blue, citecolor=blue, %
                urlcolor=blue, pagebackref=true]{hyperref}

\makeatletter

\newcommand{\arxiv}[1]{\href{http://arxiv.org/abs/#1}{{\tt arXiv:#1}}}

\newcounter{maintheorem}[equation]
\def\themaintheorem{\thesection.\@arabic \c@maintheorem}
\@addtoreset{equation}{section}
\def\theequation{\thesection.\@arabic \c@equation}

\def\theenumi{\@alph\c@enumi}
\def\theenumii{\@roman\c@enumii}

\makeatother

\theoremstyle{plain}
\newtheorem{theorem}[equation]{Theorem}
\newtheorem*{theorem*}{Theorem}
\newtheorem{lemma}[equation]{Lemma}

\newtheorem{proposition}[equation]{Proposition}

\theoremstyle{definition}

\newtheorem{innercustomthm}{Theorem}
\newenvironment{customthm}[1]
  {\renewcommand\theinnercustomthm{#1}\innercustomthm}
  {\endinnercustomthm}

\newtheorem{remark}[equation]{Remark}

\newtheorem{example}[equation]{Example}

\newtheorem{definition}[equation]{Definition}

\newtheorem{notation}[equation]{Notation}

\newtheorem{discussion}[equation]{Discussion}

\newtheorem{observation}[equation]{Observation}

\newtheorem{construction}[equation]{Construction}

\newcommand{\ov}{\overline}

\newcommand{\D}{\mathcal{D}}

\newcommand{\m}{{\mathfrak m}}

\def\to{\longrightarrow}

\DeclareMathOperator{\ann}{Ann}

\DeclareMathOperator{\LT}{LT}

\DeclareMathOperator{\HF}{HF}
\DeclareMathOperator{\Hom}{Hom}

\DeclareMathOperator{\order}{order}

\DeclareMathOperator{\Hilb}{Hilb}

\def\res{{\rm K}}

\def\sers2{{\res[\![x,y]\!]}}
\def\ser3{{\res[\![x,y,z]\!]}}

\def\pol2{{\res[x,y]}}
\def\pol3{{\res[x,y,z]}}

\makeatletter
\def\RDerChar{\mathbf{R}}
\def\RDer{\@ifnextchar[{\R@Der}{\ensuremath{\RDerChar}}}
\def\R@Der[#1]{\ensuremath{\RDerChar^{#1}}}
\makeatother
\usepackage[normalem]{ulem}

\newcommand{\nn}{\mathbf{n}}
\newcommand{\N}{\mathbb{N}}
\newcommand{\DD}{\mathfrak{D}}

\usepackage{xcolor}
\title[The Hilbert functions of local complete intersections]
{On the Hilbert function of Artinian local complete intersections of codimension three}

\author[Jelisiejew]{Joachim Jelisiejew}
\address{Faculty of Mathematics,  Informatics and Mechanics, University of Warsaw
Room 1310,  Banacha 2,  Warsaw,  Poland}
\email{jjelisiejew@mimuw.edu.pl}

\author[Masuti]{Shreedevi K. Masuti}
\address{Department of Mathematics, Indian Institute of Technology Dharwad, WALMI Campus, PB Road, Dharwad - 580011, Karnataka, India}
\email{shreedevi@iitdh.ac.in}

\author[Rossi]{M.~E.~Rossi}
\address{Dipartimento di Matematica, Universit{\`a} di Genova, Via Dodecaneso 35, 16146 Genova, Italy}
\email{rossim@dima.unige.it}

\thanks{JJ is partially supported by National Science Centre grants
2020/39/D/ST1/00132 and 2018/31/B/ST1/02857. SKM is supported by INSPIRE
faculty award funded by Department of Science and Technology, Govt. of India.
She was supported by INdAM COFOUND Fellowships cofounded by Marie Curie
actions,  Italy,  during which the work began.  MER is supported by PRIN-MIUR
2020355B8Y. This work is partially supported by the
Thematic Research Programme \emph{Tensors: geometry, complexity and quantum
entanglement}, University of Warsaw, Excellence Initiative – Research
University and the Simons Foundation Award No.~663281 granted to the Institute
of Mathematics of the Polish Academy of Sciences for the years 2021-2023.}

\keywords{Hilbert functions, complete intersection ring,  Gorenstein rings,  Macaulay's inverse system,  symmetric decomposition}
\subjclass[2010]{Primary: 13D40, 13H10; Secondary:13A30,  13J05, 13F20}

\begin{document}

\begin{abstract}
In singularity theory or algebraic geometry, it is natural to investigate possible Hilbert functions for special algebras $A$ such as local complete intersections or more generally Gorenstein algebras. 
The sequences that occur as {the} Hilbert functions of standard graded complete intersections  are well understood classically thanks to Macaulay and Stanley.
Very little is known in the local case except in codimension two.
In this paper we characterise the Hilbert functions of quadratic Artinian complete intersections of codimension three.  Interestingly we prove that a Hilbert function is admissible for such a Gorenstein ring if and only if is admissible for such a complete intersection.  We provide an effective construction of a local complete intersection for a given Hilbert function.
{We prove that the symmetric decomposition of such a complete intersection ideal is determined by its Hilbert function.}

\end{abstract}

\maketitle
\section{Introduction}

Let $(A,\m)$ be an Artinian local $K$-algebra where $K\simeq A/\m$. Its most
important \emph{numerical} invariant is the {\it Hilbert function} of its
\emph{associated graded ring} $gr_\m(A):=\oplus_{i\geq 0}\m^i/\m^{i+1}$.
Possible Hilbert functions are characterized by Macaulay's
bound~\cite[4.2.10]{BH98} and called \emph{O-sequences}. For applications, in
particular in singularity theory or algebraic geometry, it is natural to
investigate possible Hilbert functions for special algebras $A$ such as local
complete intersections or more generally Gorenstein algebras. A further recent
important motivation for this question comes from motivic homotopy theory,
where local complete intersections parametrize the algebraic cobordism
spectrum~\cite{On_the_infinite_loop_spaces}.  Determining all possible Hilbert
functions of complete intersections would be a step in analysing the
Bia{\l{}}ynicki-Birula decomposition of the local complete intersection locus
of the Hilbert scheme of points,  which is key in an approach to the Wilson space conjecture ~\cite{On_the_infinite_loop_spaces}.
\vskip 2mm
The {possible} Hilbert functions for {\it graded} complete intersections {are classically known} and depend only on the degrees  of the generators of the homogeneous ideal $I$ and on the embedding dimension \cite{Mac1916}, \cite[Theorem 4.2]{stanley}. 
In the local case  to determine the Hilbert function of a  complete
intersection  is a difficult task because  the associated graded ring does not need to be a complete intersection. 
A  characterization of the Hilbert function of a  {\it local complete  intersection of codimension two} was done by Macaulay, see  \cite{Mac04, Iar94, GHK06,GHK07,Bri77, Ber09, Kot78}.
\vskip 2mm
Complete intersection rings are in particular {\it Gorenstein rings}  and the two notions coincide in codimension two. 
In codimension three Stanley characterized {the possible Hilbert functions of \emph{graded} Gorenstein algebras by proving that they are symmetric and their first difference is an admissible Hilbert function of a codimension two Artinian algebra, see \cite{stanley}}.  {However,  it is an open question which numerical functions can occur as the Hilbert functions of \emph{local} complete intersections or, more in general, of Gorenstein local rings of codimension three, despite Buchsbaum-Eisenbud structure theorem for such rings~\cite{BE77}.} 
 This problem was stated by Iarrobino in \cite{Iar94} and it is  one of the open problems published in occasion of the Joint International Meeting of the AMS, European Mathematical Society (EMS) and Portuguese Mathematical Society (SPM)  Porto, Portugal, June  2015, workshop ``Commutative Artinian Algebras and Their Deformations''.
The first open case is when $I$ is generated by a regular sequence $(f,g,p)$ in $R=\ser3 $ the formal power series ring in three variables where $f$, $g$, $p$ {have nonzero and linearly independent quadratic parts}.  In the current paper we completely resolve this case.

\vskip 2mm
{Let us introduce formally the setup.}
An O-sequence {$h$} is a {\it{complete intersection (resp. Gorenstein) sequence}}  if $h$ is the Hilbert function of some complete intersection (resp. Gorenstein) $K$-algebra $A$ as above.
Now let $A=R/I $ be an Artinian Gorenstein ring where $R=\ser3 $   and $I \subseteq (x,y,z)^2$ an ideal of $R$ such that the Hilbert function $h$ of $A$ satisfies $h_1 = h_2 = 3$.  Write $h$ as a vector $ h= (1,3,3,h_3, ,h_4,\ldots,h_s=1)  $  where  $s$ is the {\it socle degree},  that is   the maximum integer $j$ such that $\m^j \neq 0.$ For simplicity we call an O-sequence  of the form $ h= (1,3,3,h_3, h_4,\ldots,h_s=1)$ a  {\it $(1,3,3)$  sequence}.  In this paper we will characterize the   $(1,3,3)$ sequences which are Gorenstein sequences and, in particular, those which are complete intersection sequences.

\vskip 2mm
 For an O-sequence $h$,  we set $\max h:=\max \{ h_i: i \geq 0\}$ and
\[
\Delta(h):=\max\{|h_i-h_{i-1}|:i=3,\ldots,s\}.
\]
We say that $h$ has a {\it fall} by $m$ at $i$ (resp.~ {a} {\it fall} $m$) if $h_{i+1} = h_i-m$ (resp.  $h_{i+1}=h_{i}-m$ for some $i$).  
By Macaulay's bound, if $h$ is an $(1,3,3)$ sequence,  then $h_3 \leq 4$,  and $h$ {is unimodal.}
The smallest $t$ such that $h_{t+1} < h_t$
is called the {\it peak position} of $h$.
The following is the first main result of the paper.  
  \begin{customthm}{1}
\label{Equi:GorAndCIIntro}  Let $h$ be a $(1,3,3)$ sequence.  Then the following statements are equivalent:
 \begin{enumerate} 
 \item[(i)]   $h$ is a complete intersection sequence;
  \item [(ii)]  $h$ is a Gorenstein sequence;
  \item [(iii)] $h$ satisfies one of the following conditions:\\
(I) $  h_3 \leq 3;$ \\
(II) $h_3 =4$ and $ \Delta(h)=1$; \\
(III) $h_3=4,~\Delta(h)=2$ and $h$ has a {\bf unique} fall by {\bf two} at the {\bf peak} position.  
In this case, $h$ is of the following form:
 $$h_i=\begin{cases}
i+1 & \mbox{ for } i\le d-2\\
 d & \mbox{ for } i =d-1,d,  \dots, d+r-1\\
 d-2 & \mbox{ for } i=d+r\\
 {h_{i-1}\mbox{ or } h_{i-1} - 1} & \mbox{ for } i\geq d+r+1
\end{cases}$$
for some integers $d = \max h \ge 4 $ and  $r \ge 0.$
\end{enumerate} 
       \end{customthm}
\vskip 2mm      
 
It is clear that $(i)$ implies  $(ii) $.  We prove that $(ii) $ implies
$(iii)$ in Section~\ref{sect3} and $(iii) $ implies $(i)$ in Section~\ref{sect4}.  
In each case  we give constructive methods and examples: see Theorem \ref{thm:h3Atmost3} for $h$ of type (I), see Theorem \ref{1334} and  Example \ref{Example:ConsCI} for type (II) and (III). 
We remark that Artinian Gorenstein algebras with Hilbert function $(1, 3, 3, 3, 1)$ have been classified in \cite{Jel17}.


\vskip 2mm
 {After proving Theorem~\ref{Equi:GorAndCIIntro}, we} discuss the
realizable symmetric decompositions of $(1,3,3)$ sequences.  It is well-known
that the Hilbert function of graded Artinian Gorenstein $K$-algebra is
symmetric,  that is,  $h_i=h_{s-i}$ for all $i=0,1\ldots,s.$ This is no longer
true for local Artinian Gorenstein $K$-algebras.  In \cite{Iar94}  Iarrobino
proved that the Hilbert function of Artinian Gorenstein $K$-algebra admits a
symmetric decomposition (see Section \ref{Subsec:Q-decomp}).  It is natural to
ask what  {are the possible} symmetric decompositions  {for} a
complete intersection ideal.  In two variables, where  {being} Gorenstein is
equivalent to  {being a} complete intersection,  Iarrobino showed that the Hilbert function determines the symmetric decomposition~\cite[Theorem 2.2]{Iar94}.  We observe that this is the case also for complete intersection ideals with Hilbert function as $(1,3,3)$ sequence.  This is the second main result of the paper.

  \begin{customthm}{2}
  \label{thm:Q-deco-133}
  Let $h$ be a $(1,3,3)$ complete intersection sequence.  Then the following
  hold for ideals with Hilbert function $h$:
  \begin{asparaenum} 
  \item  {there is a unique possible symmetric decomposition for} a complete intersection ideal,
  \item when $\Delta(h) = 2$,  {there is a unique possible symmetric decomposition for} a Gorenstein ideal,
  \item \label{thm:Q-deco-Gor} {when $\Delta(h) = 1$ there} are two possible
      symmetric decompositions for a Gorenstein ideal.
  \end{asparaenum}
    \end{customthm}

In fact,  we give a precise description of symmetric decompositions for Gorenstein and complete intersection ideals in Propositions \ref{Prop:Q-decOf133} and \ref{Prop:Q-decOf1334}.  Theorem \ref{thm:Q-deco-133} is one of the few known results where complete intersections can be distinguished numerically from Gorenstein algebras.
In~\cite{Iar83} Iarrobino asks for the dimension of the family of complete
intersections with a  given Hilbert function.   {While our results give
partial structure theorems for ideals with the Hilbert functions above}, this
question {remains} open even in those cases.  {In general, it is} now understood to be very non-trivial: while deforming a local complete intersection is trivial, keeping track of the fact that it is local and of its Hilbert function is hard.
\vskip 2mm

We have used the computer algebra systems \cite{Macaulay2, Singular} and the library \cite{Elias} 
for various computations in this paper.

\section{Preliminaries}
In this section we fix the notations and present the basic tools that will be used in the paper.  We also give the references for a more complete background.  

\subsection{Macaulay's inverse system}
\label{Subsec:InvSyst}
As a consequence of Matlis duality, Macaulay's inverse system plays an important role in the construction of  a 
Gorenstein $K$-algebra.

Let $R=K[\![x_1,\ldots,x_m]\!]$ and $\D := K_{DP}[X_1, \ldots ,X_m]$, which we
view as a polynomial ring with $\deg X_i=1$ for all $i$. The vector space $\D$ has a structure of $R$-module
via the \emph{contraction} action, as follows. For $\nn=(n_1,\ldots,n_m),
\nn'=(n_1',\ldots,n_m') \in \N^m,$ $x^{\nn}=x_1^{n_1}\cdots x_m^{n_m} \in R$
and $X^{\nn'}= X_1^{n_1'}\cdots X_m^{n_m'} \in \D$ we define
\begin{equation}\label{Eqn:DividedPower}
  {x^{\nn} \circ X^{\nn'}} :=\begin{cases}
\ X_1^{n'_1-n_1}\cdots X_m^{n'_m-n_m}& \mbox{ if } \nn' \geq \nn\\
0 & \mbox{ otherwise,}
          \end{cases}
\end{equation}
where by $\nn' \geq \nn$ we mean that $n_i' \geq n_i$ for all $i=1,\ldots,m$.
The contraction action is extended $K$-linearly to $R$ and $\D$.
For the experts we note that $\D$ has another, more canonical, ring
    structure of a divided power ring,
    see \cite[Theorem 3]{Nor74} or \cite[Appendix A]{ik99}. To make the paper
    more accessible, we ignore this ring structure as we will not need it.

    On the one hand, to an ideal $I \subseteq R$ we associate an $R$-submodule of $\D$
\[
 I^\perp:=\{F \in \D:f \circ F=0 \mbox{ for all }f\in I\}.
\]
This submodule of $\D$ is called the \emph{Macaulay's inverse system of $I$}. On the
other hand,  for an $R$-submodule $W$ of $\D,$ its
    \emph{annihilator} is an ideal of $R$
defined as follows:
\[
 \ann_R(W):=\{f \in R:f \circ F=0 \mbox{ for all }F \in W\}.
\]
For $F_1,\ldots,F_k\in \D,$ we write $\langle F_1,\ldots,F_k \rangle$ for the
$R$-submodule of $\D$ generated by $F_1,\ldots,F_k$ and
$\ann_R(F_1,\ldots,F_k)$ for $\ann_R(\langle F_1,\ldots,F_k\rangle)$. For $F
\in \D,$ we set $A_F:=A/\ann_R(F)$ and call it the \emph{apolar algebra}
of $F$.
\vskip 2mm

By Matlis duality,  $R/I$ is Artinian Gorenstein  $K$-algebra if and only if $ (R/I)^\vee \cong I^\perp $ is a cyclic $R$-submodule of $\D$ (see \cite[Section 3.2]{BH98}). 
Emsalem in \cite[Proposition 2]{Ems78} and Iarrobino in \cite{Iar94}, based on the work of Macaulay in \cite{Mac1916}, gave a more precise description of the inverse system of Artinian Gorenstein $K$-algebras:

\begin{proposition}  
\label{Prop:Emsalem-Iarrobino}
	There is a one-to-one correspondence between ideals $I$ such that $R/I$ is an Artinian Gorenstein local  $K$-algebra of socle degree $s$ and cyclic $R$-submodules of $\D$ generated by   a nonzero {polynomial} of degree $s.$ The correspondence is defined as follows:
	\begin{eqnarray*}
		\left\{  \begin{array}{cc} I \subseteq R \mbox{ such that } R/I \mbox{ is an Artinian } \\
			\mbox{Gorenstein local $K$-algebra of }\\
			\mbox{ socle degree $s$ } 
		\end{array} \right\}  \ &\stackrel{1 - 1}{\longleftrightarrow}& \ 
		\left\{ \begin{array}{cc} \langle F \rangle \subseteq \D \mbox{ submodule generated by a nonzero}\\
			 \mbox{ polynomial $F$ of degree }  s \\
		\end{array} \right\} \\
        I\subseteq R \ \ \ \ \ &\longrightarrow& \ \ \ \ \ I^\bot  \ \ \ \ \ \ \ \ \qquad \ \ \ \  \ \ \ \qquad \ \ \ \  \ \ \ \\
        \ann_R(W)\subseteq R \ \ \ \ \  &\longleftarrow& \ \ \ \ \  W \ \ \ \ \ \ \ \ \ \ \ \ \ \ \ \  \qquad \qquad \ \ \ \  \ \ \
	\end{eqnarray*}
\end{proposition}

In this paper we are interested in Artinian $K$-algebras of codimension three. Hence for simplicity we set $R=K[\![x,y,z]\!]$ and $\D:=K_{DP}[X,Y,Z].$

\subsection{Hilbert function and symmetric decomposition}
\label{Subsec:Q-decomp}
The Hilbert function of a local ring $A$ with maximal ideal $\m$ and residue field $K \simeq A/\m$ is defined as follows: for every $t \ge 0$
$$
\HF_A(t) =   \dim_\res \left(\frac{\m^t}{\m^{t+1}}\right).
$$

Then $\HF_A(t)$ is equal to the minimal number of generators of the ideal $\m^t$ and  the Hilbert function of the local ring $A$ is the Hilbert function of the standard graded algebra $$gr_{\m}(A)=\oplus_{t\ge 0}\ \m^t/ \m^{t+1}.$$ This algebra is called {\it{the associated graded ring of the local ring}}  $(A,\m)$.  {It has a geometric interpretation in} the case when $A$ is the localization at the origin O of the coordinate ring of an affine variety $V$ passing through O: then $gr_{\m}(A)$ is the coordinate ring of the {\it tangent cone } of $V$ at O, which is the cone composed of all lines that are the limiting positions of secant lines to $V$ in O.

\vskip 2mm
It is possible to compute the Hilbert function of $A=R/I $ via the inverse system.  Namely, with the previous notation, 
\begin{equation} \label{H2}
\HF_{R/I}(i)= \dim_K    (I^{\perp})_i 
\end{equation}
where 
\begin{equation*} \label{eqn:DualOfI}
 (I^{\perp})_i := {\frac{ I^{\perp} \cap \D_{\le i} +  \D_{< i}}{ \D_{< i}}}.
\end{equation*}
 
 \newcommand{\mari}[1]{\textcolor{red}{#1}}
  
  {If $A$ is Artinian, then $\HF_{A}(i)=0$ for all $i>s$ where $s$ is the socle degree of $A$. In accordance with the notation of Section 1, we recall that the Hilbert function is represented by a vector $h$ where $h_i=\HF_{A}(i).$}

It is well-known that the Hilbert function of an Artinian graded Gorenstein $K$-algebra is symmetric. { This is no longer} true in the local case. The problem comes from the fact that, in general, the associated algebra $G:=gr_{\m}(A)$ of {an Artinian} Gorenstein local algebra $A$ is no longer Gorenstein. However, in  \cite{Iar94} Iarrobino proved that the Hilbert function of an Artinian Gorenstein local $K$-algebra $A$ admits a ``symmetric'' decomposition. To be more precise,   
consider a filtration  of $G $  by a descending sequence of ideals:
$$ G = C(0) \supseteq C(1) \supseteq \dots \supseteq  C(s)=0,  $$
where 
$$C(a)_i:=\frac{(0:\m^{s+1-a-i}) \cap \m^i}{(0:\m^{s+1-a-i})\cap \m^{i+1}}.$$
Let $ Q(a)= C(a)/C(a+1),$ then $\{Q(a) : a=0,\ldots,s-1\}$
is called the {\it symmetric decomposition} of the associated graded ring $G$.  We write $\DD(a)$ for the Hilbert function of $Q(a)$ and say that $\mathfrak{D}_A=(\DD(0),\DD(1),\ldots,\DD(s-2))$ is the symmetric decomposition of $A$ or of $\HF_{A}.$   The latter is an abuse of language: $\DD(a)$ cannot be reconstructed from $\HF_A$. We simply write  $\DD$ for $\DD_A$ if $A$ is clear from the context.  By a symmetric decomposition of an ideal $I$ we mean the symmetric decomposition of ${R/I}.$ We have
\begin{eqnarray*}
\HF_A(i)=\dim_K G_i =\sum_{a=0}^{s-1} \DD(a)_i.  
\end{eqnarray*} 
We say that a symmetric decomposition $\mathfrak{D}$ of a given Gorenstein sequence is realizable if there exists a Gorenstein ideal with the symmetric decomposition $\mathfrak{D}$.
Iarrobino \cite[Theorem 1.5]{Iar94} proved that if $A=R/I$ is a Gorenstein
local ring then for all $a=0,\ldots,s-1,$ $Q(a)$ is a reflexive graded
$G$-module, up to a shift in degree: $\Hom_K(Q(a)_i,K) \cong Q(a)_{s-a-i}.$
Hence the Hilbert function of $Q(a)$ is symmetric about $\frac{s-a}{2}.$
Moreover, he showed that $Q(0) = G/C(1) $ is the unique graded Gorenstein
quotient of $G$ with socle degree $s$. Let  $f={f_s}+{\text{ lower degree terms...} } $ be a polynomial in $\D$ of degree $s$ where  ${f_s} $ is the homogeneous part of degree $s$   and consider $A_f$  the corresponding  Gorenstein local $K$-algebra. Then, $Q(0) \cong R/\ann_R({f_s})$, see \cite[Proposition 7]{Ems78} and \cite[Lemma 1.10]{Iar94}. See also \cite{IM20} for a  useful discussion on the symmetric decomposition.

\subsection{Hilbert function and standard bases}
\label{Subsec:LT}
Let  $ I$ be an ideal of $R=K[\![x_1,\ldots,x_m]\!]$  and consider the local ring $A=R/I$ whose maximal ideal is   $\m:= (x_1, \ldots ,x_m)/I$. Since the Hilbert function of a local ring $A$ is the same as that of the associated graded ring $gr_{\m}(A), $   it will be useful to recall the presentation of this standard graded algebra.  Let $P=K[x_1,\ldots,x_m]$ be the standard graded polynomial ring.   For every power series  $f\in R\setminus \{0\}$ we can write $f={f_v+f_{v+1}} +\cdots$, where ${f_v}$ is not zero and ${f_j}$ is a homogeneous polynomial of degree $j$ in $P$ for every $j\ge v.$
We say that $v$ is \emph{the order} of $f$,  {we denote $f_v$} by $f^*$ and call it \emph{the initial form} of $f.$
 If $f=0$ we agree that its order is $\infty.$
 It is well known that $gr_{\m}(A) {\simeq} P/I^*$,  where $I^*$ is the homogeneous ideal of  $P  $ generated by the initial forms of the elements of $I.$
 A set of power series $f_1,\cdots,f_r\in I$ is a \emph{standard basis} of $I$ if $I^*=(f_1^*,\cdots,f_r^*). $ 
 Every ideal $I$ has a standard basis and every standard basis is a system of generators of $I$.
  However,  not every system of generators is a standard basis.
  To determine a standard basis of a given ideal of $R$ is {not straightforward, the task is similar in flavour to computing a Gr\"obner basis.}
 We denote by $T$ the set of terms or monomials of $P$; let $\tau$ be a term ordering in $T.$ 
We define a new total order ${\ov{\tau}}$ on $T$ in the following way:
given $m_1, m_2\in T$
we let
$m_1 >_{{\ov{\tau}}} m_2$ if and only if $\deg(m_1) < \deg(m_2)$ or $\deg(m_1) = \deg(m_2)$ and $m_1 >_{\tau} m_2$.
\smallskip
Given $f\in R, $  there is a monomial which is the maximum of the monomials in
the support of $f$ with respect to ${\ov{\tau}}$: namely,  since the support of $f^*$ is a finite set, we can take  the maximum with respect to $\tau$ of the elements of this set. This monomial is called the leading monomial of $f$ with respect to ${\ov{\tau}}$ and is denoted by $\LT_{{\ov{\tau}}}(f).$ By definition, we  have   
$$\LT_{{\ov{\tau}}}(f)=\LT_{\tau}(f^*).$$
As usual,  we define the leading term ideal associated to an ideal $I\subset R$ as the   monomial ideal $ \LT_{\ov{\tau}}(I) $  generated in $P$ by $\LT_{\ov{\tau}}(f)$ with $f$ running in $ I, $ see \cite{GP08}.
\smallskip
In \cite{Ber09} a  set $\{ f_1, \dots , f_r\} $ of elements of $I$  is called an
{\it{enhanced}}  standard basis of $I$  if the corresponding leading terms generate $\LT_{\ov{\tau}} (I).$ Every enhanced standard basis is also a
standard basis, but the converse is not true: an example is given by the ideal $I=(x^2+y^2,xy+y^3)$ in the power series ring $K[\![x,y]\!].$
In \cite{GP08} an enhanced standard basis of $I$  is simply called a standard basis.
We have $\LT_{\ov{\tau}}(I)  =\LT_{\tau}(I^*)$,  see \cite[Proposition~1.5]{Ber09}, so that
\[
 \HF_{R/I}=\HF_{P/I^*}=\HF_{P/\LT_{{\ov{\tau}}}(I)}.
\]
In the theory of enhanced standard {bases} a crucial result is Grauert's
Division theorem \cite[Theorem 6.4.1]{GP08}. It claims the following: given a set of  formal power series $f, f_1,\cdots, f_m \in R$
there exist power series $q_1,\dots,q_m,r \in R$ such that $f=\sum_{j=1}^mq_jf_j+r$ and, for all $j=1,\dots,m$,
\begin{enumerate}
\item[(1)]
No monomial of $r$ is divisible by $\LT_{{\ov{\tau}}}(f_j)$,

\item[(2)]
 $\LT_{{\ov{\tau}}}(q_jf_j)\le \LT_{{\ov{\tau}}}(f)$ if $q_j\neq 0.$
\end{enumerate}

\noindent
Using this we can obtain in the formal power series ring all the properties of Gr\"{o}bner bases analogous to those proved in the classical case for polynomial rings.

\section{necessary conditions  for  Gorenstein sequences }\label{sect3}

In this section  we give a necessary condition for a $(1,3,3)$ sequence  to be Gorenstein.  In particular we prove $(ii) \implies (iii)$ of Theorem \ref{Equi:GorAndCIIntro}.  By Macaulay's bound we know that if $h$ is a $(1,3,3)$ sequence, then $h_3 \leq 4. $ If  $h_3 \leq 3,$ the condition   $(iii)$ of Theorem \ref{Equi:GorAndCIIntro} is clear,   hence in this Section from now onwards {we} assume that $h$ is a $(1,3,3,4)$ sequence, that is $h_3=4.$

Throughout this section we fix a term order $\tau$ on the set of monomials of $P=K[x,y,z]$ such that $z>y>x.$

\begin{lemma}
\label{Lemma:reductionOFfghStronger}
Let $h$ be a $(1,3,3,4)$ sequence   and let $A=R/I$ be an Artinian algebra with the Hilbert function $h$.
Then there exist $f,g,p \in I$ such that up to a change of coordinates the
quadratic parts of $f$, $g$, $p$ are $xz$, $yz$, $z^2$, respectively.
\end{lemma}
\begin{proof}
    Since $h_2 = 3$,  there exist elements $f, g, p\in I$ whose quadratic
    parts are linearly independent. Consider the graded ideal $I^*\subseteq P$
    and its graded subideal $J = (f^*, g^*, p^*)$. We have $\HF_{P/J}(2) = 3$
    and $\HF_{P/J}(3) \geq \HF_{P/I^*}(3) = \HF_{R/I}(3) = 4$. By Macaulay's
    bound, we have $\HF_{P/J}(3)\leq 4$, hence $\HF_{P/J}(3) = 4$. Let $J'$ be
    the ideal generated by $J$ and linear forms $\ell\in P_1$, if any, such
    that $P_1\ell \subset J$. Thanks to this last condition, we have
    $\HF_{P/J'}(i) = \HF_{P/J}(i)$ for all $i\geq 2$. In the language of~\cite{cjn13} the
    ideal $J'$ is $2$-saturated,  so by \cite[Lemma~2.9]{cjn13} we obtain $H_{P/J'}(1) = 2$.
    Hence $J'_1\neq 0$ and up to coordinate change we assume $J'_1 = Kz$, so
    that $xz, yz, z^2\in J$. Since $J$ is generated by $(f^*, g^*, p^*)$
    up to coordinate change we have $f^*=xz$, $ g^*=yz$, $ p^*=z^2$.
\end{proof}

Below in this section we assume that the change of coordinates from Lemma~\ref{Lemma:reductionOFfghStronger} is already done.
Our key idea in order to classify the  Gorenstein sequences of type $(1,3,3,4)$ is to analyse the inverse system of $ I \cap K[\![x,y]\!]$ in $K_{DP}[X, Y]$. This
reduces the problem to codimension two,  where  we can take advantage of the rich literature (see  \cite{Ber09}, \cite{Bri77}, \cite{GHK06}, \cite{HW20}, \cite{Iar94}, \cite{MR15}, \cite{Mac04}).  We recall the following well-known result on the characterization of Hilbert functions of codimension two complete intersections which we will use frequently in the paper:
the O-sequence $h=(1,2,h_2,\ldots,h_{s-1},h_s=1)$ is a complete intersection (equivalently,  Gorenstein) sequence if and only if 
\begin{eqnarray}
\label{Eqn:codimension2Gsequence}
|h_i-h_{i-1}| \leq 1 \mbox{ for all } i=1,\ldots,s.
\end{eqnarray}
This result was first obtained by Macaulay in \cite{Mac04},  and well-described by Brian\c{c}on in \cite{Bri77} and Iarrobino in \cite{Iar77}.  See \cite[Theorem 4.6B]{Iar84},  \cite[p. ~23]{Iar94} for a recent version of this result.
We refer  to  \cite[Theorem 2.6]{Ber09} for the construction of Gorenstein ideals with the Hilbert function $h$ satisfying \eqref{Eqn:codimension2Gsequence}.

\begin{proposition}
\label{Prop:HFOfJ}
Let $A=R/I$ be an Artinian algebra such that the ideal $I$ contains  elements $f,g,p$ with $\LT_{\ov{\tau}}(f)=xz$, $\LT_{\ov{\tau}}(g)=yz$, $\LT_{\ov{\tau}}(p)=z^2$.   
Set $S:= K[\![x,y]\!]$ and $J=I \cap S.$ Then $S/J$ has the Hilbert function 
 \[
  \HF_{S/J}(i)=\begin{cases}
              2 & \mbox{ if }i=1\\
              \HF_{R/I}(i) & \mbox{ if } i \neq 1.
              \end{cases}
 \]
\end{proposition}
\begin{proof}
   {Recall the notation $P = K[x,y,z]$.} We know that 
\[
 \HF_{R/I}= \HF_{P/\LT_{\ov{\tau}}(I)}.
\]
Similarly,
\[
 \HF_{S/J} =\HF_{K[x,y]/\LT_{\ov{\tau}}(J)}.
\]
We have  $xz,yz,z^2 \in\LT_{\ov{\tau}}(I).$  Therefore by Grauert's Division theorem (see Section \ref{Subsec:LT}),  we deduce 
\[
 \LT_{\ov{\tau}}(I)=(xz,yz,z^2) + \LT_{\ov{\tau}}(J).  
\]
Thus we have an exact sequence
\begin{eqnarray*}
 0 \to (\LT_{\ov{\tau}}(J)+(z))/\LT_{\ov{\tau}}(I) \to K[x,y,z]/\LT_{\ov{\tau}}(I) \to K[x,y,z]/(\LT_{\ov{\tau}}(J)+(z))  \to 0.
\end{eqnarray*}
Let $M:=(\LT_{\ov{\tau}}(J)+(z))/\LT_{\ov{\tau}}(I)$. Since $P/(\LT_{\ov{\tau}}(J)+(z)) \cong K[x,y]/\LT_{\ov{\tau}}(J),$ we get
\begin{equation}
\label{Eqn:ExactSeq}
 \HF_{R/I}=\HF_{S/J}+\HF_{M}.
\end{equation}
Now
\begin{eqnarray*}
 M &=&(\LT_{\ov{\tau}}(J)+(z))/((xz,yz,z^2) + (\LT_{\ov{\tau}}(J))) \\
 &\cong & (z)/((z) \cap \LT_{\ov{\tau}}(J)+(xz,yz,z^2))\\
 &\cong & (z)/z((\LT_{\ov{\tau}}(J):_P z)+(x,y,z)) = (P/(x,y,z))(-1).
\end{eqnarray*}
Therefore from \eqref{Eqn:ExactSeq} we get the required result. \qedhere
 \end{proof}

The following is a crucial result in establishing the numerical criteria for a sequence to be Gorenstein.

\begin{proposition}\label{Prop:classificationOfArtinianGorenstein}
Let $R/I$ be an Artinian Gorenstein local ring  with the Hilbert function of type $(1,3,3,4). $  Set $S:=K[\![x,y]\!]$ and $J=I \cap S.$   Then there exist $F,G \in K_{DP}[X,Y]$ such that 
    \begin{enumerate}
     \item \label{it:fstG} $x\circ G, y \circ G\in \langle F\rangle$;
     \item \label{it:sndG} $J^{\perp} = \langle F, G\rangle.$
   \end{enumerate}
 In this case,   $\HF_{S/J}(i)  =
      \HF_{R/I}(i)$   for all $i \neq 1.$
   
   In particular, if $h$ is a Gorenstein sequence of type $(1,3,3,4), $  
   then $\Delta(h) \leq 2,$ and there exists  at most  one position with a  {fall} by two. 
 \end{proposition}
\begin{proof}
   By Lemma~\ref{Lemma:reductionOFfghStronger} the ideal $I$ contains elements of the form $f=xz-U$, $g=yz-V$,
    $p=z^2-W$,  where $U, V, W\in R_{\geq 3}$.   We may assume that $U,V, W
    \in S_{\geq 3}$ by taking into account the particular form of
    $f^*,g^*,p^*$ and the fact that $(x,y,z)^k \subseteq I$ for $k \gg
    0$ {: as
        long as some of $U$, $V$, $W$ depend on $z$, we subtract an  appropriate
    combination of $f$, $g$, $h$. This procedure stops as it increases the
order and the order $\geq k$ part can be discarded}.  
{This is essentially the procedure described by Grauert's Division Theorem quoted  {above, which implies that each of} $U$, $V$ and $W$ can be written as a combination of $f,g,p$ plus a remainder  $r$ where no monomial  of $r$ is divisible by the leading monomials of $f,g,p,$ that is,  by   $xz,yz, z^2.$}

Since $R/I$ is Gorenstein, there exists $T\in K_{DP}[X, Y, Z]$ such that
    $I^\perp=\langle T \rangle $.
    Write
    \begin{equation}\label{eq:Tform}
        T = T_0 + Z\cdot T_1 + Z^2 \cdot T_2 +  \ldots + Z^n\cdot T_n +
        \ldots ,
    \end{equation}
    where $T_i\in
    K_{DP}[X, Y].$
   Since $p \circ T=0$, we obtain $z^2 \circ T = W \circ T$, so that
    \begin{equation}\label{eq:evencontainment}
        T_{n+2} = W \circ T_n\quad\mbox{for all } n\geq 0.
    \end{equation}
    Therefore, $\langle T_{n+2}\rangle \subseteq \langle T_n\rangle$ for all $n\geq 0$. In
    particular, $\langle T_{2n}\rangle \subseteq \langle T_{2n-2} \rangle \subseteq  \ldots  \subseteq
    \langle T_0 \rangle$ and similarly $\langle T_{2n+1}\rangle \subseteq
    \langle T_1\rangle$ for all $n\geq 0$.
    Moreover, $f\circ T = 0$ and $g \circ T = 0$ translate into the conditions
    \begin{equation}\label{eq:oddcontainment}
        x \circ T_{n+1} = U \circ T_n,\quad y \circ T_{n+1} = V \circ
        T_n\quad\mbox{for all }n\geq 0.
    \end{equation}
    Let $F := T_0$ and $G := T_1$. Then by \eqref{eq:oddcontainment} 
    \begin{eqnarray*}
     x \circ G = U \circ F \in \langle F \rangle \mbox{ and } y \circ G= V \circ F  \in \langle F \rangle.
    \end{eqnarray*}
Since $J=I \cap S,$ $J = \ann_R(T) \cap S=\ann_S (T)= \bigcap_{i \geq 0}
    \ann_S(T_i)$.
    As $\langle T_{2n} \rangle \subseteq \langle F\rangle$ and $\langle
    T_{2n+1} \rangle \subseteq \langle G\rangle$, we have $J =
    \ann_S(F)\cap \ann_S(G)$, hence $J^{\perp} = \langle F, G\rangle$ as required.
    By  {\cite[Theorem~2.3]{Ber09}} we have $\Delta \HF_{S/J}\leq 2$.
    By~\eqref{it:fstG}-\eqref{it:sndG} it follows that $\HF_{S/J}(i) = \HF_{S/ \ann_S(F)} (i) $   for all $i$ except one position.
   {By the above cited result, we have $\Delta\HF_{S/\ann_S(F)}\leq 1$,
    so $\HF_{S/J}$ has at most one fall by two.}
    By Lemma \ref{Lemma:reductionOFfghStronger} and Proposition
    \ref{Prop:HFOfJ},   we have $\HF_{R/I}(i)= \HF_{S/J}(i) $  for all $i \ge
    2$, so  {also $\HF_{R/I}$ satisfies $\Delta\HF_{R/I} \leq 2$ with a
    single possible position with a  fall by two.}
    \end{proof}

    {The converse of Proposition \ref{Prop:classificationOfArtinianGorenstein} is  not true.}  That is, let  $F$ and $G$ in $K_{DP}[X,Y]$ be  such that $x \circ G, ~y \circ G \in \langle F \rangle $,  and assume there exists an ideal $I$ in $R$ with the  Hilbert function of type  $(1,3,3,4)$  such that $I \cap S =\ann_S(F,G) $.  Then $I$  is not necessarily  a Gorenstein ideal.  
The following example illustrates this.  
\begin{example}
\label{Example:ConvGor}
Let $I=(xz,yz,z^2-y^3,x^4)$ be an ideal in $R.$ Then $R/I$ is an Artinian local ring with the Hilbert function $h=(1,3,3,4,2,1)$. Consider $F=X^3Y^2, ~G=Y^3 \in K_{DP}[X,Y].$ Then $x \circ G,$ $y \circ G \in \langle F \rangle .$ Also $I \cap S=(x^4,xy^3,y^4)=\ann_S(F,G).$  But $I$ is not a Gorenstein ideal  by Buchsbaum-Eisenbud structure theorem since $I$ is $4$-generated.  
\end{example}
\vskip 2mm
 
The following {simple} example shows that $h$  is not necessarily a  Gorenstein sequence even if $h$ is a $(1,3,3,4)$-sequence with $\Delta(h)=2.$ Recall that in codimension two $\Delta(h)=1 $ was sufficient for being a Gorenstein sequence. 
\begin{example}
\label{Example:jump2}
 Consider the O-sequence $h=(1,3,3,4,3,1)$. Here $\Delta(h)=2,$ but is not a Gorenstein O-sequence since it does not 
 admit a symmetric decomposition.   
\end{example} 

Now we investigate at which position the {fall} by $2$ can occur if $h$ is a $(1,3,3,4)$-sequence with  $\Delta(h)=2.$ For this first we look for the unique position at which $\HF_{S/J}$ and $\HF_{S/\ann_S(F)}$ in Proposition \ref{Prop:classificationOfArtinianGorenstein} possibly differ.

\begin{lemma}
\label{lemma:NumCriForGorSeqNew}
Let $k \geq 3$ and let  $J \subseteq I$ be ideals in $S:=K[\![x,y]\!]$ such that $I$ is a Gorenstein ideal and
\[
    \HF_{S/J}(i)=\begin{cases}
 \HF_{S/I}(i) & \mbox{ for } i \neq k\\
 \HF_{S/I}(i)+1 & \mbox{ for } i=k.
\end{cases}
\]
Then $I$ has a minimal generator of order $k$ and $k$ is the peak position of $\HF_{S/J}$.
\end{lemma}
\begin{proof}
Since $J \subseteq I,$ we have $\dim_K I/J=\dim_{K} S/J - \dim_K S/I=1$.  Hence $\m_S (I/J)=0$ where $\m_S=(x,y)S$.  Therefore $\m_S I \subseteq J.$ Since $I$ is a complete intersection ideal,  $\dim_K I/\m_S I =2.$ Therefore 
\[
\dim_K J/\m_S I = \dim_K I/\m_S I - \dim_K I/J=2-1=1.
\]
Choose $g_1 \in J \setminus \m_S I.$ For $g \in J,  $ let $\overline{g}$ denote the image of $g$ in $J/\m_S I$.  Then $ K\overline{g_1} = J/\m_S I.$ Therefore $J=(g_1)S + \m_S I.$ Choose $g_2 \in I \setminus \m_S I$ such that $\overline{g_1}$ and $\overline{g_2} $ are $K$-linearly independent elements of $I/\m_S I$.  Then by Nakayama's lemma $I=(g_1,g_2).$ Also,  
$$J=(g_1) + \m_SI=(g_1)+(x,y)(g_1,g_2)= (g_1) +(x,y)(g_2) .$$ 
Recall that the Hilbert function of $S/J$ is the Hilbert function of $K[x,y]/J^*$.  
{Without loss of generality we may assume that the initial forms of $g_1$ and $g_2$ are $K$-linearly independent.} 
Then $I^*=(g_2^*)+J^*.$   Since the Hilbert functions of $K[x,y]/I^*$ and $K[x,y]/J^*$ are the same except at the position $k$,  $g_2^* $ is a form of degree $k$ in $ I^*  $.  Therefore $g_2$ has order $k$.   \\
Let $h_I:=\HF_{S/I},~h_J=\HF_{S/J}$ and $d:=\max h_J.$ 
If $\max h_I=d-1,$ then $k=d-1$,  and is the peak position of $h_J.$ Suppose $\max h_I=d.$ 
Then there exist elements $g_1$ of order $d$ and $g_2$ of order $t+1$ in $I$ such that $I=(g_1,g_2)$ where $t$ is the peak position of $h_I.$ Hence by above $k=d$ or $k=t+1$.  If $k=d,$ then $h_J(d)=d+1$ and hence $\max h_J=d+1,$ a contradiction.  Therefore $k=t+1.$ Since $t$ is the peak position of $h_I$ and $\Delta(h_I) \leq 1 $ by \eqref{Eqn:codimension2Gsequence},  $h_{I}(t+1)=d-1,$ and so $h_{I}(t+2) \leq d-1$.  This implies that $h_J(k+1) \leq d-1$ while $h_J(k)=h_I(k)+1=d.$ Hence $k$ is the peak position of $h_J$.   
\end{proof}

We know that a  $(1,3,3,4)$ sequence is unimodal by Macaulay's {bound:}  $h$ is of the form 
 $h=(1,3,3,4,5,\ldots,d,d,h_t,\ldots1)$ where $d:=\max h$ and $h_t <d.$ In the following proposition we observe that such a   Gorenstein sequence with $\Delta(h)=2$ has a unique fall by two at $t-1$. 
 
  \begin{proposition}
 \label{Prop:NumCriGorSeq}
Let $h$ be a $(1,3,3,4)$ Gorenstein sequence with $\Delta(h)=2$.  Let  $d:=\max h$ and let $r $ be a nonnegative integer such that $d$ is repeated $r+1$ times in $h.$   Then
$$h_i=\begin{cases}
i+1 & \mbox{ for } i\le d-2\\
 d & \mbox{ for } i =d-1,d, \dots, d+r-1\\
 d-2 & \mbox{ for } i=d+r\\
 {h_{i-1} \mbox { or } h_{i-1} -1 } &  \mbox{ for } i\ge d+r+1. 
\end{cases}$$
 \end{proposition}
\begin{proof}
    Let $h$ be a $(1,3,3,4)$ Gorenstein sequence with $\Delta(h)=2$.  By Proposition \ref{Prop:classificationOfArtinianGorenstein}  $h$ has a unique fall by two,  say at position $m.$  By Proposition \ref{Prop:classificationOfArtinianGorenstein} there exist $F,G \in K_{DP}[X,Y]$ such that $\HF_{S/\ann_S(F)}$ and $\HF_{S/\ann_S(F,G)}$ differ,  possibly,  at one position, say $k$, and moreover $\HF_{S/\ann_S(F,G)}$ and $h$ differ only at position $1$. Let $J =\ann_S(F,G)$.
    Hence \[
        \HF_{S/J}=(1,2,3,4,\ldots,h_{m-1},h_m,h_{m+1}=h_m-2,\ldots,1).
\]
By~\eqref{Eqn:codimension2Gsequence} we have $\Delta({\HF_{S/\ann_S(F)}}) = 1$, thus $\HF_{S/\ann_S(F)}$ and $\HF_{S/J}$ differ at position $m$. Since they differ in at most one position, we have
\[
    \HF_{S/\ann_S(F)}=(1,2,3,4,\ldots,h_{m-1},h_m-1,h_{m+1}=h_m-2,\ldots,1).
\]
Therefore $k=m.$ By Lemma \ref{lemma:NumCriForGorSeqNew} the position $k
$ is the peak position of $\HF_{S/J}$,  and so $m=k$ is the peak position of $h.$ Now it is easy to check that $h$ has the required form.
 \end{proof}

\begin{proof}
    [Proof of $(ii) \implies (iii)$ of Theorem \ref{Equi:GorAndCIIntro}:] Let $h$ be a $(1,3,3)$ Gorenstein sequence.  Then $h_3 \leq 4$  because $h$ is an $O$-sequence.  If $h_3\leq 3,$ {then the result is clear}.  Let  $h_3=4.$ Then $\Delta(h) \leq 2$ by Proposition \ref{Prop:classificationOfArtinianGorenstein}.  Hence  $\Delta(h)=1 $ or, by Proposition \ref{Prop:NumCriGorSeq}, $h$ is a $(1,3,3,4)$ sequence with a  unique fall by two at the peak position.  
\end{proof}

We now give an example of a non-Gorenstein sequence using  {the above proven
implication in} Theorem~\ref{Equi:GorAndCIIntro}.  This example is a particular case of a more general result  \cite[Corollary 3.12]{IM21}.

\begin{example}
\label{Example:IM}
Consider the O-sequence $h=(1,3,3,4,5,4,4,2,1)$.  As a consequence of \cite[Corollary 3.12]{IM21} it follows that $h$ is not a Gorenstein sequence.  
We show that $h$ is not a Gorenstein sequence using Theorem \ref{Equi:GorAndCIIntro}.    
Here $d=\max h = 5$ and hence the only possibility for a  {fall by} $2$ is the
position $4$. The fall occurs elsewhere, and hence this  sequence is not a Gorenstein sequence.

Similarly the sequence  $(1,3,3,4,3,3,3,1)$ is not  Gorenstein because here a {fall} by $2$ can occur only at the position $3$ by  Theorem \ref{Equi:GorAndCIIntro}.  See \cite[p.100]{Iar94} for an alternative argument.  
\end{example}

\section{Construction of $(1,3,3)$ complete intersection
sequences}\label{sect4}
In this section   we will construct an  ideal which is a complete intersection with the Hilbert function $h$ in each of the following cases:
 \begin{enumerate}
     \item  $(1,3,3,\leq 3)$ sequence,  {that is a $(1,3,3)$ sequence with $h_3 \leq 3$, }
 \item $(1,3,3,4)$ sequence.  
 \end{enumerate} 
 
 This proves $(iii) \implies (i)$ of Theorem \ref{Equi:GorAndCIIntro}.  In the case $h_3 \leq 3,$ we give a  polynomial $F \in K_{DP}[X,Y,Z]$ such that $\ann_R(F)$ is a CI ideal with the  {Hilbert function} $h$.  If $h_3=4,$ by reducing to the codimension two case,  we give an explicit algorithm to construct a  CI ideal with the  {Hilbert function} $h$.

 \subsection{Construction of a complete intersection ideal with the {Hilbert function} $(1,3,3,\leq 3)$} In the following theorem we prove that any $(1,3,3)$ sequence $h$ with $h_3 \leq 3$ is admissible for a complete 
intersection. In fact, we give an explicit $F \in \D$ such that $A_F$ is a complete intersection with the Hilbert function $h.$ 
Notice that $h$ is a $(1,3,3,\leq 3)$ sequence implies that $h$ is of the form $\left(1, 3, 3, {3,  \ldots , 3}, {2, \ldots
        ,2}, {1 \ldots 1}, 1\right)$  by Macaulay's condition.

\begin{theorem}
 \label{thm:h3Atmost3}
 Let 
        \[ h := \left(1, 3, 3, \underbrace{3,  \ldots , 3}_u, \underbrace{2, \ldots
        ,2}_v, \underbrace{1, \ldots, 1}_w, 1\right)\] be an O-sequence. Then
        there exists a complete intersection $A=R/(f,g,p)$   with the Hilbert function $h.$ More precisely, 
         \begin{enumerate}
            \item (trivial case) Suppose that $u = v = w = 0.$ Then $\ann_R(XYZ)=(x^2,y^2,z^2)$ is a complete intersection and $R/\ann_R{ (XYZ) }$ has the Hilbert function $h = (1, 3, 3, 1)$.
            \item Suppose that at least one of the $u, v, w$ is nonzero. Then 
            \[\ann_R(F)=\left(yz - x^{u + v + w + 2},\,xz - y^{u + v + 2},\, xy- z^{u+2}
                \right)\] where $ F := X^{u + v + w + 3} + Y^{u + v + 3} + Z^{u + 3} + XYZ
                $              
                is a complete intersection and $A_F$ has the Hilbert function $h.$ 
        \end{enumerate}
\end{theorem}
\begin{proof}
 The case $u = v = w = 0$ is trivial. Hence we assume that at least one
        of the $u$, $v$, $w$ is nonzero.  Let $A := u + v + w + 2$, $B := u + v + 2$ and $C := u + 2$. Then
        $A \geq B \geq C\geq 2$ and $A > 2$.
        Note that the polynomial $F = X^{A + 1} + Y^{B + 1} + Z^{C + 1} + XYZ$ has the same
        Hilbert function as $F_0 = X^{A+1} + Y^{B+1} + Z^{C+1}$. Indeed, for
       any operator $\partial$  {of order at least two} the
        polynomial $\partial\circ F -\partial\circ F_0$ 
        is at most linear, so the leading forms of $\partial\circ F$ and $\partial\circ F_0$ 
        agree. {We use~\eqref{H2} to conclude that the Hilbert
        functions of $A_{F}$ and $A_{F_0}$ agree}.  {We check directly
        that} $A_{F_0}$ has the Hilbert function $h$, $A_{F}$ also has the Hilbert function $h.$ 
       {(The algebra $A_{F_0}$ is not a complete intersection; we use it
    only to compute the Hilbert function.)}

        It is easy to verify that
        \begin{equation*}\label{eq:elsone}
           I:= (x^A - yz,\ y^{B} - xz,\ z^{C} - xy) \subseteq \ann_R(F).
            \end{equation*}
            {We claim that $I=\ann_R(F)${, so $I$ is a complete
            intersection}. We prove this below.}
    First we show that $\HF_{R/I}(i) \leq h_i$ for all $i=0,\ldots,A+1.$ 
        Note that
        \[
            x^{A+2} \equiv x^2 yz =
            (xz)(xy) \equiv y^{B}z^C =
            y^{B-C}(yz)^C \equiv y^{B-C}x^{AC} \mod I,
        \]
        so that $x^{A+2}\left(1 -  y^{B-C}x^{AC - A - 2} \right) \in I$. We have $AC - A - 2\geq 2A - A - 2 = A-2 > 0$ and $B - C\geq 0$. 
        Thus $y^{B-C}x^{AC - A - 2} \in (x,y,z).$ Therefore $1 -  y^{B-C}x^{AC - A - 2}$ is
        invertible and hence
        \begin{equation}\label{eq:elstwo}
            x^{A+2} \in I.
        \end{equation}
        Moreover,
        \begin{equation}\label{eq:elsthree}
            xyz \equiv x^{A+1} \equiv y^{B+1} \equiv
            z^{C+1} \mod I.
        \end{equation}
        Let $I^*$ be the ideal generated by the initial forms of elements
        of $I$. Then $\HF_{R/I} = \HF_{P/I^*}$ where $P=K[x,y,z]$.
        The equations \eqref{eq:elstwo},
        \eqref{eq:elsthree} imply that each monomial may be reduced modulo
        $I^*$ to a monomial of one of the forms $x^a$, $y^b$ or $z^c$, for $a \leq
        A+1$, $b \leq B$ and $c \leq C$. Therefore $\HF_{P/I^*}(i) \leq h_i$ for all $i=0,\ldots,A+1$ {and so}
       $\dim_K(R/I) \leq \sum_{i \geq 0} h_i= \dim_K A_F.$ Since $I \subseteq \ann_R(F),$ this   
        proves the claim and hence the result.
\end{proof}

\begin{remark}
    The reason to distinguish $u = v = w = 0$ case is that the ideal $(yz -
    x^{2},\ xz - y^2,\ xy - z^2)$  is not a complete intersection.
   {This case is special in that there is no pure power of higher
    degree, so the division procedure hidden in the proof of
    Theorem~\ref{thm:h3Atmost3} does not stop.}
    \end{remark}
    
    We remark that there are Gorenstein $K$-algebras with the Hilbert function a $(1,3,3,\leq 3)$ sequence   {which are not} complete intersections.
\begin{example}
Let $F=X^4+Y^3+Z^3$ and $I=\ann_R(F)=(xy,xz, yz, x^4-z^3,y^3-z^3).$ Then $R/I$ is Gorenstein with the Hilbert function $(1,3,3,1,1)$,  but $I$ is not a complete intersection.
\end{example}

\subsection{Construction of a  complete intersection ideal with the  {Hilbert function} of $(1,3,3,4)$ type} 
First we prove the following proposition to construct a complete intersection ideal starting from $F,G \in  K_{DP}[X,Y]$ satisfying certain conditions,   compare Proposition~\ref{Prop:classificationOfArtinianGorenstein}.

\begin{proposition}
 \label{Prop:PropertiesOfJ}
 Let $F, G\in K_{DP}[X,Y] $ be  minimal  generators of the $S$-submodule
 $\langle F,G \rangle$ and  further assume that
 \begin{equation}\label{eq:cond}
 x \circ G,~y\circ G \in \m_S^2 \circ F
 \end{equation}
where $\m_S=(x,y)S$ is the maximal ideal of $S=K[\![x,y]\!].$ 
Let $J:=\ann_S(F,G).$ Then there
exists a complete intersection ideal $I$ in $R$ such that $I \cap S=J.$
 Moreover, if $S/J$ has the Hilbert function $h',$ then $R/I$ has the Hilbert function $h$ where
 \[
      h_i=\begin{cases}
        h_i' & \mbox{ if } i \neq 1\\
         3 & \mbox{ if } i=1.
     \end{cases}
 \]
\end{proposition}
\begin{proof}
Let $J=\ann_S(F,G)$. 
 Then $J = \ann_S(F)\cap \ann_S(G)$. Let
     $\omega_{S/J}$ be the canonical module of $S/J$. Then $\omega_{S/J} \cong J^\perp$ is the
     $S$-submodule of $K_{DP}[X, Y]$ generated by $F$ and $G$.  First, we do a
     change of coordinates. Write $y \circ G = a_2
     \circ F$ for $a_2 \in \m_S^2$. Let $a_2 = xa_2' + ya_2''$ {for some $a_2',~a_2''\in \m_S$} and
     $G' = G - a_2'' \circ F$. Then the inverse
     systems $\langle F,G\rangle$ and $\langle F,G'\rangle$ are equal. Moreover, 
     $$y \circ G' = y\circ G -(ya_2'') \circ F  = a_2 \circ F - (ya_2'') \circ F = (a_2 - ya_2') \circ F= (xa_2')\circ F.$$ 
     Hence  by replacing $G$ by $G'$ we may assume that $a_2$ is divisible by
     $x$ and so $a_2 = xa_2'$ for $a_2'\in \m_S$.

     The relations $x\circ G, y \circ G\in
     \m_S^2\circ F$ give rise to syzygies of $\omega_{S/J}$. As $G\not\in \langle
     F\rangle$, these syzygies are minimal, so the minimal resolution of
     $\omega_{S/J}$ is
     \[
         0\to S\to S^3 \overset{N}{\to} S^2 \to 0,
     \]
     where $N$ has, up to coordinate change, the form
     \[
         N = \begin{pmatrix}
             d_{11} &  d_{12}& -xa_2' \\
             d_{21} & -x & y \\
         \end{pmatrix}
     \]
     with $d_{12}\in \m_S^2$.  {As noted above, $G\not\in\langle
         F\rangle$. By~\eqref{eq:cond} also $F\not\in \langle G\rangle$. Hence
         $d_{11}, d_{21}\in \m_S$. By further column operations one could
         assume $d_{21} = 0$, but this will not be important for us.}
     Let $d_{12} = xd_{12}' + yd_{12}''$ where $d_{12}', d_{12}''\in \m_S$.
     Let $U_1 = -\frac{1}{2}(d_{12}' + d_{21})$, $U_2 =
     -d_{12}''$, $V_1 = -a_2'$, $V_2 =
     \frac{1}{2}(d_{12}' - d_{21})$. By construction $U_1, U_2, V_1, V_2\in
     \m_S$. Let further
     \[
         U=xU_1+yU_2,\qquad V = xV_1 + yV_2,\qquad W = U_2V_1 + V_{2}^2 - d_{11}.
     \]
     By construction we have $U, V\in \m_S^2$ and $W\in \m_S$.
     By duality,  the {minimal} resolution of $S/J$ is
    \[
        F_{\bullet}: 0 \to S^2 \overset{N^T}{\to} S^3 \to S \to 0
    \]
    where  \[
         N^T=\left( \begin{array}{cc}
            d_{11}=U_2V_1+V_2^2-W &d_{21}= -U_1-V_2\\
            d_{12}=-U+xV_2& -x \\
           -xa_2'=xV_1 & y
         \end{array} \right).
     \]
    Therefore $J$ is the ideal generated by the $2 \times 2$  minors of $N^T.$ Thus 
    $$J=(-yU+x V, -UV_1-VV_2+yW, -UU_1-VU_2+xW ).$$ 
   Let us {define} an ideal $I\subset K[\![x, y, z]\!]$ by $I = (f:=xz - U, g:=yz
    - V, p:=z^2 -W)$.
We have
\begin{eqnarray}
    -yU+x V &=& yf-xg \label{fstgen}\\
     -UV_1-VV_2+yW&=&fV_1+(z+V_2)g-yp\label{thirdgen}\\
 -UU_1-VU_2+xW&=& (z+U_1)f+gU_2-xp. \label{sndgen}
 \end{eqnarray}
Hence $J \subseteq I \cap K[\![x, y]\!]$. In particular $(x, y)^{t} \subset I$ for $t \gg 0$. Since also $z^2
-W\in I$ and $W\in \m_S,$ we have $z^{2t} \in I$.  Hence we conclude that $R/I $ is Artinian.  Thus $I$ is a complete intersection.\\
{\bf Claim:} $J=I \cap K[\![x, y]\!]$.\\
    {\bf Proof of claim:} Clearly,   $J \subseteq I \cap K[\![x, y]\!]$.  Suppose that $j = \alpha f + \beta g + \gamma p\in K[\![x,
y]\!]$ for some $\alpha,\beta,\gamma \in R$. We want to prove that $j\in J$.  
 By adding suitable multiples of~\eqref{sndgen} and~\eqref{thirdgen}
to $j$, we may assume $\gamma\in K[\![z]\!]$. But then $0 \equiv j \equiv z^2\cdot
\gamma\mod (x, y)$, hence $\gamma = 0$. Thus $j = \alpha f + \beta g$.
After adding a multiple of~\eqref{fstgen} to $j$, we may assume that $\beta\in
K[\![y, z]\!]$.
Assume $j\neq 0$.
Fix a term ordering $\tau$ on the set of monomials of $P=K[x,y,z]$ such that 
$x<y<z.$ Then with the notation fixed in Section \ref{Subsec:LT} it follows
that $ \LT_{{\ov{\tau}}}(f) = xz$ and $\LT_{{\ov{\tau}}}(g)= yz.$ Hence $\LT_{{\ov{\tau}}}(\alpha f)$ is
divisible by $x$ while $\LT_{{\ov{\tau}}}(\beta g) \in K[y, z]$.  Thus $\LT_{{\ov{\tau}}}(\alpha f)+\LT_{{\ov{\tau}}}(\beta g)  \neq 0.$ Therefore $\LT_{{\ov{\tau}}}(j)  = \LT_{{\ov{\tau}}}(\beta g) $ or  
$\LT_{{\ov{\tau}}}(j)  = \LT_{{\ov{\tau}}}({\alpha f})$.  Now by assumption $j \in K[\![x,
y]\!], $ hence
$\LT_{{\ov{\tau}}}(j) \in K[x, y]$ while $\LT_{{\ov{\tau}}}({\alpha f}), \LT_{{\ov{\tau}}}({\beta g})   \in zK[x,y,z]$.
Thus $\LT_{{\ov{\tau}}}(j)  = 0$, so $j = 0$ is indeed an element of $J$.

    Therefore,  by Proposition \ref{Prop:HFOfJ},  $R/I$ has the Hilbert function $h.$  \qedhere
\end{proof}

 {In the following Lemma~\ref{ref:existenceOfG} we will need the following observation on the symmetric
decomposition. Since it is disconnected from the remaining part of the proof of the Lemma and unimportant for us
except for the proof, we leave it as a remark.}
\begin{remark}\label{ref:square}
    Suppose $h$ is a Hilbert function of a Gorenstein algebra $B$ with embedding dimension $h_1 = 2$.
    Then the symmetric decomposition of $h$ is
    unique~\cite[Theorem~2.2]{Iar94}.
    Let  $t = \max(i\ |\ h_t  \geq 2)$. 
    It can be directly verified that $Q(a)_1 \neq 0$ if and only if $a = 0$ or
    $a = s - t - 1$, where $s$ is the socle degree of $B$.  
    This implies that $Q(a)_1 = 0$ for all $a \geq s - t$, so $C(a)_1 = 0$ for all $a\geq s
    -t$. It follows from the definition
    of $C(a)_1$ that $(0 : \m_B^{\delta})\cap \m_B = (0 : \m_B^{\delta}) \cap \m_B^2$ for all
    $\delta \leq t$, where $\m_B$ is the maximal ideal of $B$.
\end{remark}

 \newcommand{\spann}[1]{\langle #1 \rangle}
 \begin{lemma}\label{ref:existenceOfG}
     Let $F\in K_{DP}[X,Y]$ be such that $A_F$ has  {the} Hilbert function
     \[
         h = (1, 2, 3, 4 \ldots,d-1 , d, d,  \ldots,  d, d, h_k=d-1, h_{k+1},
         h_{k+2}, \ldots,1 ),
     \]
     where $d=\max h$.  Assume $d\geq 3$.    Then
     there exists $G\in K_{DP}[X,Y]$ such that $S/\ann_S(F, G)$ has  {the} Hilbert
     function
     \[
         h' = (1, 2, 3,  \ldots , d-1,d, d,  \ldots , d, d, d, h_k' =d, h_{k+1},
         h_{k+2}, \ldots,1 ).
     \]
     Moreover, $F$ and $G$ satisfy the conditions of Proposition~\ref{Prop:PropertiesOfJ}.
 \end{lemma}
 \begin{proof}
  Let $\m_S=(x,y)S.$   We know that $\ann_S(F)$ is a complete intersection $(f, g)$ where
     $\order(f) = d$ and $\order(g) = k$ and $g\not\in (f) \cap \m_S^k + \m_S^{k+1}$.
     Consider the ideal $I = (f, xg, yg)$. 
     The vector space $(f, g) / I$ is nonzero and spanned
     $K$-linearly by $g$, hence it is one-dimensional. Therefore, also the space
    $I^{\perp} / (f, g)^{\perp} = I^{\perp} / \spann{F}$
    is nonzero and one-dimensional. Choose $G\in I^{\perp} \setminus \spann{F} $ such that the image of $G$ spans
    $\frac{I^{\perp}}{\spann{F}}$.  Then on the one hand, we have $G\not\in \spann{F},$ so
    $\spann{G} \cap
    \spann{F} \subset \m_S \circ \spann{G}$. On the other hand, $\spann{G} \subset
    I^{\perp},$ which implies that $\dim_{K} \spann{G}/(\spann{F}\cap \spann{G}) \leq 1$. Since
    $\dim_{K} \spann{G} / \m_S \circ \spann{G} = 1$, this shows that $\spann{F}\cap
    \spann{G} = \m_S \circ \spann{G}$, so $x\circ G, y\circ G\in \langle
    F\rangle$. Since $\langle F, G\rangle = I^{\perp}$, it follows that $I =
    \ann_S(F, G)$.

    We now compute the Hilbert function. Since $g$ has order $k$, we
    have $xg, yg\in \m_S^{k+1}$.
    Since $g\not\in (f) \cap \m_S^k+ \m_S^{k+1}$ we see that the natural surjection
    \[
        \frac{\m_S^k}{\m_S^{k+1} + I \cap \m_S^k} = \frac{\m_S^k}{\m_S^{k+1} + (f) \cap \m_S^k} \to \frac{\m_S^k}{\m_S^{k+1} + (f, g) \cap \m_S^k} =  \frac{\m_S^k}{\m_S^{k+1} + (f) \cap \m_S^k+(g)}
    \]
    is not an isomorphism. This shows that the Hilbert functions of $S/(f,g)$ and
    $S/I$ differ in position $k$. Since the difference of
    dimensions of these algebras is one,  $k$ is the only position where they
    differ, which
    proves that $S/I$ has the Hilbert function $h'$. From
    the equality of $h$ and $h'$ for arguments higher than $k$,  we deduce
    $\langle F,G \rangle \cap K_{DP}[X,Y]_{\leq i} + K_{DP}[X,Y]_{< i}=
    \langle F \rangle \cap K_{DP}[X,Y]_{\leq i} + K_{DP}[X,Y]_{< i}$
    for all $i >k.$  Hence we may assume that $\deg(G) \leq k.$ Indeed,  suppose $c:=\deg(G)>k.$ Then by the equality $\langle F,G \rangle \cap K_{DP}[X,Y]_{\leq c} + K_{DP}[X,Y]_{< c}= \langle F \rangle \cap K_{DP}[X,Y]_{\leq c} + K_{DP}[X,Y]_{< c},$ we deduce that  $G \in \langle F \rangle \cap K_{DP}[X,Y]_{\leq c} + K_{DP}[X,Y]_{< c}.$ Write $G=u \circ F+H$ where $u \circ F \in K_{DP}[X,Y]_{\leq c}$ and $H \in K_{DP}[X,Y]_{< c}.$
 {We have $I^\perp/\langle F \rangle =K \{G\} =K \{ H\}$, so we can
replace $G$ by $H$ and hence assume that it has degree $<c$}.  Continuing like
this we may assume that $G$ has degree $\leq k$. {In fact, its leading
form is a nonzero element of degree $k$ even modulo $\langle F\rangle$, hence $\deg(G) = k$.}

    Let $t = \max(i\ |\ h(i) \geq 2)$. Since $d\geq 3,$ we have $t \geq k$.
    Since $\deg(G)\leq k,$ we have $\deg(x\circ G), \deg(y\circ G)\leq k-1$.  
{The map $f: S/\ann_S(F) \to \langle F\rangle $ given by $f(\ov{u}) = u \circ F$ is an isomorphism where $\ov{u}$ denotes the image of $u$ in $S/\ann_S(F).$ Since $x\circ G \in \langle F\rangle$,  there exists $\alpha \in S$ such that $f(\ov{\alpha}) = \alpha \circ F=x \circ G$.  
Then $\ov{\alpha} \in \m_F \cap (0 : \m_F^{k})$, where $\m_F$
    is the maximal ideal of $S/\ann_S(F).$ Indeed
    $\deg(x \circ G) \leq k-1$ implies that $\m_S^k \circ (x \circ G)=0$, which gives that $\m_S^k \circ (\alpha \circ F)=0.$ Hence $\m_S^k \alpha \in \ann_S(F)$, that is,  $\ov{\alpha} \in (0:\m_F^k)$.  }
    By Remark~\ref{ref:square} we have
    $\m_F \cap (0 : \m_F^{k}) = \m_F^2 \cap (0 : \m_F^k)$,  so that  $\ov{\alpha} \in \m_F^2. $  Hence $x \circ G=\alpha  \circ F \in \m_S^2\circ F$.  Similarly,  $y\circ G\in \m_S^2\circ F$.
    Hence $F$ and $G$ indeed satisfy the
    conditions of Proposition~\ref{Prop:PropertiesOfJ}.
    This concludes the proof.
 \end{proof}

 \begin{remark}
     Lemma~\ref{ref:existenceOfG} can be made constructive for certain $F$
     as follows. Fix any $h$ as in the Lemma and take $F = \sum_{i=1}^{s} \ell_i^{[e_i]}$; where $e_1\geq
     e_2 \geq  \ldots \geq e_s$ exist and are determined uniquely by requiring
     that $h$ is the Hilbert function of $S/\ann_S(F)$, while
     $\ell_i\in K_{DP}[X, Y]_1$ for all $i$ {are $K$-linearly independent}.  We then see that $\langle F\rangle \cap
     K_{DP}[X, Y]_{\leq k-1}  = \langle \ell_1^{[k-1]}, \ldots ,
     \ell_s^{[k-1]}\rangle$ (modulo $K_{DP}[X,Y]_{<k-1}$).   {If $s \geq k+1$,} then $G = \ell_s^{[k]}$ satisfies the conditions of Proposition~\ref{Prop:PropertiesOfJ}. 
 \end{remark}

    \begin{theorem} \label{1334}
        Let $h$ be a $(1,3,3,4)$ sequence with {$\Delta(h)=1$} or {$\Delta(h)=2$} having  {a} unique fall by two at the peak position.   
        Then there exists a complete intersection ideal $I$ having the Hilbert function $h$.
    \end{theorem}
    \begin{proof}
      Let $d:=\max h$ be repeated $r+1$ times in $h.$  Let $h'$ be the Hilbert function $h$ with a single change so that
        $(h')_1 = 2$. Let $h''$ be the Hilbert function $h'$ with a single
        change at the peak position so that the last $d$ is replaced by $d-1$.
        Then $h''$ is unimodular and $\Delta(h'') = 1$,
        hence by \eqref{Eqn:codimension2Gsequence}  there exists a complete intersection algebra $S/\ann_S(F)$
        with Hilbert function $h''$.

        We will now construct $G$ such that $F$, $G$ satisfy the
        conditions  of
        Proposition~\ref{Prop:PropertiesOfJ} and, moreover, $S/\ann_S(F,G)$
        has the Hilbert function $h'$.

        Consider first the case $r\geq 1$. Then $\max h'' = d$ and we take
        $G$ as in
        Lemma~\ref{ref:existenceOfG}.
        It remains to consider the case $r = 0$. In this case
        $h'' = (1, 2, 3,  \ldots , d-1, d-1, \ldots )$.  It follows that
        $\spann{F}$ contains the space of degree $\leq d-2$ polynomials,  but
        there exists a polynomial $G\in K_{DP}[X,Y]$ of degree $d-1$ not in $\spann{F}.$ Then $x \circ G,~y\circ G \in \m_S^2 \circ F $  by the comparison between $h' $ and $h".$ Such
        a $G$ satisfies all our conditions.   {Moreover,  $S/\ann_S(F,G)$ has the Hilbert function $h'$. }
        Hence by Proposition ~\ref{Prop:PropertiesOfJ} there exists a complete intersection ideal $I$ with the Hilbert function $h.$
    \end{proof}

We remark that the proof of the above theorem is constructive.  
That is,  given a $(1,3,3,4)$ sequence $h$ with $\Delta(h)=1$ or {$\Delta(h)=2$ having}  {a} unique fall by two at the peak position,   
the above result gives an algorithm to construct a complete intersection ideal $I$ with the Hilbert function $h$.  In fact,  {the} following is  {an} algorithm for this.  Recall that if $h$ is a $(1,3,3)$ sequence with $h_3 \leq 3,$ then Theorem \ref{thm:h3Atmost3} gives an explicit construction of a complete intersection ideal $I$ with the Hilbert function $h.$
\vskip 2mm
\hrule
\vskip 0.5mm
\hrule
\vskip 2mm
\noindent \textcolor{blue}{\underline{\bf Algorithm:}} \\
\noindent \underline{{\bf Given:}} a $(1,3,3,4)$ sequence $h$ with {$\Delta(h)=1$} or with {$\Delta(h)=2$ having} a  unique fall by two at the peak position.\\ 
\underline{\noindent {\bf Aim:}} To construct a complete intersection ideal $I$ with the Hilbert function $h$.\\
\underline{{\bf Step 1:}} Let $d=\max h$ and $t$ be the peak position of $h$.  Define
\[
    h'_i=\begin{cases}
h_i & \mbox{ if } i \neq 1 \mbox{ or } t\\
2 & \mbox{ if } i=1\\
d-1 & \mbox{ if } i=t.
\end{cases}\]
\vskip 2mm
\noindent \underline{{\bf Step 2:}} Construct $F \in K_{DP}[X,Y]$ such that $\ann_S(F)$ has the Hilbert function $h'$. 
\vskip 2mm
\noindent \underline{{\bf Step 3:}} Let $\ann_S(F)=(f,g)$ with $\order(f) \leq \order(g).$  
\vskip 2mm
\noindent \underline{{\bf Step 4:}} If $\max h' = d-1,$ then choose any $G \notin \langle F \rangle$ of degree $d-1.$ Suppose $\max h'=d.$ Then consider $I=(f,xg,yg)$ and find $G \in I^\perp$ such that $G \notin \langle F \rangle$ and $\deg(G) \leq t$.  Then $x \circ G, ~y \circ G \in \m_S^2 \circ F.$ Replace $G$ by $G'$,  if necessary,  so that  $y \circ G=a_2 \circ F$ and $x $ divides $a_2$. 
\vskip 2mm
\noindent \underline{{\bf Step 5:}} Write $y \circ G=(x a_2') \circ F.$\\  \\
\underline{{\bf Step 6:}}  Find the syzygy matrix in a minimal $S$-free resolution of $S/\ann_S(F,G)$.  It will be of the form
 \[
         N = \begin{pmatrix}
             d_{11} &  d_{12}& -xa_2' \\
             d_{21} & -x & y \\
         \end{pmatrix}^t.
     \]
\vskip 2mm     
\noindent \underline{{\bf Step 7:}} Write $d_{12}=xd_{12}'+yd_{12}''$ .  Set $U=xU_1+yU_2,$ $V=xV_1+yV_2$ where 
 $U_1 = -\frac{1}{2}(d_{12}' + d_{21})$, $U_2 =
     -d_{12}''$, $V_1 = -a_2'$, $V_2 =
     \frac{1}{2}(d_{12}' - d_{21})$,   and $W = U_2V_1 + V_{2}^2 - d_{11}$. \\ \\
\underline{{\bf Step 8:}} Then $I = (xz - U, yz- V, z^2 -W)$ is an Artinian complete intersection ideal with the Hilbert function $h.$

\vskip 2mm     

\hrule
\vskip 0.5mm
\hrule
\vskip 2mm
We illustrate this algorithm by the following example.

\begin{example}
\label{Example:ConsCI}
Consider the O-sequence $h=(1,3,3,4,2,1)$ which is a $(1,3,3,4)$ sequence with {$\Delta(h)=2$ having}  {a} unique fall by two at the peak position.   Let us find a complete intersection ideal with the Hilbert function $h$.   Define $h'=(1,2,3,3,2,1).$ We first find a complete intersection ideal with the Hilbert function $h'$  (see for instance \cite{Ber09}, \cite{MR15} and \cite{HW20} for the possible constructions).  Consider $F =
 X^3Y^2.$ Then $A_F$ has the Hilbert function $h'$.  Here $\max h'=3=\max h-1.$ Hence we choose $G \notin \langle F \rangle$ of degree $3.$ 
Let $G = Y^3$.   Then $G \notin \langle F \rangle$ and $x \circ G,~y \circ G\in \m_S^2 \circ F $.  Moreover,  $S/\ann_S(F,G)$ has the Hilbert function $(1,2,3,4,2,1)$.  The syzygy
 matrix of $J: = \ann_S(F, G) = (x^4, y^4, y^3x)$ is
 \[
     \begin{bmatrix}
         y^3 & 0 & -x^3\\
         0 & -x & y
         \end{bmatrix}^t.
 \]
 Following Proposition~\ref{Prop:PropertiesOfJ}, we put $U = 0$, $V = -x^3$, $W
 = -y^3$, so that $I = (xz, yz+x^3, z^2 + y^3)$ and indeed this is a
 complete intersection ideal with the Hilbert function $h = (1, 3, 3, 4, 2, 1)$.
\end{example}

 \section{Realizable  symmetric decompositions of $(1,3,3)$ Gorenstein sequences}
 In this section we discuss the symmetric decomposition of $(1,3,3)$ Gorenstein sequences and prove Theorem \ref{thm:Q-deco-133}.  
We know that every Gorenstein sequence admits  a symmetric decomposition,  but in general the symmetric decomposition is not unique. 
In two variables, where Gorenstein is 
equivalent to complete intersection, the Hilbert function determines the symmetric decomposition. 
We show that this is the case also for the complete intersection rings with  a $(1,3,3)$ sequence.  
Moreover,  we describe all possible symmetric decompositions of Gorenstein ideals.  

\begin{proposition}
 \label{Prop:Q-decOf133}
Let 
\[ 
h := \left(1, 3, 3, \underbrace{3,  \ldots , 3}_u, \underbrace{2, \ldots
       ,2}_v, \underbrace{1, \ldots, 1}_w, 1\right) 
       \]
       be an  O-sequence. Let $I$ be a Gorenstein ideal with  {the} Hilbert function $h$. Then
\begin{asparaenum} 
\item if $I$ is additionally a complete intersection, then its symmetric
    decomposition is unique and given by $\mathfrak{D}_1=(\DD(0),\ \DD(w),\ \DD(w+v))$, where
\begin{align*}
    \DD(0)& =(1,1,\ldots,1,1),\\
    \DD(w)&=(0,1,1,\ldots, \DD(w)_{u+v+2}=1),\\
    \DD(w+v)&=(0,1,1,\ldots,\DD(w+v)_{u+2}=1),
\end{align*}
with the understanding that  {if $w$ is zero, then $\DD(0)$, $\DD(w)$ are
    replaced by a single vector which is their sum, if $v$ is zero then the
    same happens for $\DD(w)$, $\DD(v+w)$, and if $w=v=0$ then the symmetric
decomposition consists of a single $\DD(0)$ which is the sum of the three.}
\item If  {$\Delta(h) = 2$}, then the symmetric decomposition of $I$ is $\mathfrak{D}_1.$ 
\item\label{it:decOf133Third} If {$\Delta(h)=1$}, then the symmetric decomposition of $I$ is either
    $\mathfrak{D}_1$ or  $\mathfrak{D}_2$, where $\mathfrak{D}_2$ is the
    following: let $\delta=(0,1)$ and $(\DD(0),\DD(1),\ldots,\DD(s-2))$ be the
    unique symmetric decomposition of $h-\delta.$ Then 
\[
\mathfrak{D}_2=(\DD(0),\DD(1),\ldots,\DD(s-2)+\delta).
\]
\end{asparaenum}
\end{proposition}
 The decompositions are realizable in appropriate classes.
 The quotient by the complete intersection ideal in Theorem \ref{thm:h3Atmost3} has symmetric
 decomposition $\mathfrak{D}_1$. If $h-\delta$ is as
 in~(\ref{it:decOf133Third}) and $F' \in K_{DP}[X,Y]$ is such that
 $A_{F'}$ has Hilbert function $h-\delta$, then the Gorenstein
 algebra $A_{F'+Z^2}$ has symmetric decomposition $\mathfrak{D}_2$.
\begin{proof}
We prove all three points together.
 {As remarked above, $\mathfrak{D}_1$ and $\mathfrak{D}_2$ are
realizable.}
We prove that there are no possible symmetric decompositions for $h$ other than $\DD_1$ and $\DD_2$. Take any decomposition $\mathfrak{D}$.  Observe that any $(1,3,3)$ O-sequence is unimodal.  
Also,  for any $a$, the partial sum of the first $a$ rows of $\mathfrak{D}$ is an O-sequence,  and hence 
in particular,  
is unimodal in our case.   Let $a$ be the smallest integer such that $\sum_{i=0}^{a} \mathfrak{D}(a)_1 = 3$.  Consider $\mathfrak{D}({<}a) := \sum_{i=0}^{a-1} \mathfrak{D}(a)$. If $\max \mathfrak{D}({<}a) = 2$ then by unimodality of $\mathfrak{D}({<}a) + \mathfrak{D}(a)$ we see that $\mathfrak{D}(a) = (0,1, \ldots ,1)$. It follows also that $h = \mathfrak{D}({<}a) + \mathfrak{D}(a)$ and hence $\mathfrak{D}({<}a)$ is uniquely determined as the unique decomposition of $h - \mathfrak{D}(a)$. Similar argument works for $\max \mathfrak{D}({<}a) = 1$. In both cases we recover $\mathfrak{D} = \mathfrak{D}_1$. In the case $\max \mathfrak{D}({<}a) = 3$ it follows that $\mathfrak{D}({<}a)_2 = 3$. By unimodality of $\mathfrak{D}({<}a)$ this vector agrees with $h$ for all arguments greater than one. Hence $a = s-2$ and so $\mathfrak{D} = \mathfrak{D}_2$.
 {By construction of $\mathfrak{D}_2$, we have $\Delta(\sum \mathfrak{D}_2) = 1$,
hence this case cannot occur for $\Delta(h) = 2$.}

    We claim that the decomposition $\mathfrak{D}_2$  is not realized by a
    complete intersection ideal.   Indeed,  suppose that $\mathfrak{D}_2$ is
    realized by a complete intersection ideal $I\subseteq R$.  Since the
    $(s-2)$-th part of $\mathfrak{D}_2$ is non-zero,  by
    \cite[Proposition~4.5]{cjn13} the algebra $R/I$ is isomorphic to
    $R/\ann_R(F)$ where $F=F'+Z^2$ and $F' \in K_{DP}[X,Y].$ We have
    $xz,~yz,~z^2-\sigma \in \ann_R(F)$, where $\sigma \in K[\![x,y]\!]$ is
    such that $\sigma \circ F'=1$.  Since $\ann_R(F)$ is a complete
    intersection ideal, these form its minimal generating set. This is a
    contradiction because any minimal generator $g$ of  $\ann_{K[\![x,y]\!]}(F')$ is in $\ann_R(F) \setminus (xz,yz,z^2-\sigma).$ Hence $\mathfrak{D}_2$ is not realized by a complete intersection ideal.  

This proves that if $I$ is a complete intersection ideal with the Hilbert
function $h,$ then its symmetric decomposition is determined by $h$, namely is
$\mathfrak{D}_1$.
\end{proof}

We illustrate the above discussion by the following example.

  \begin{example}
\label{Example:IarQue1}
Let $h=(1,3,3,2,1).$ Then $h$ is a complete intersection sequence by
 {Proposition~\ref{Prop:Q-decOf133} and, in its language, has $(u,v,w) =
(0,1,0)$}.  Here $h$ admits two symmetric decompositions each of which is {realizable} by a Gorenstein ideal, yet only one by a complete intersection ideal. 
The decomposition 
\[
\mathfrak{D}_1=(\DD(0)=(1,2,2,2,1),\DD(1)=(0,1,1))
\]
is realized by a complete intersection $A_F$ where $F=X^4+Y^4+Z^3+XYZ.$ The  decomposition
\[
\mathfrak{D}_2=(\DD(0)=(1,2,3,2,1),\DD(2)=(0,1))
\]  
is realized by $A_F$ where $F=X^2Y^2+Z^2.$ See \cite[Discussion 3.7]{MR18}.
\end{example}

Now we discuss realizable symmetric decompositions of $(1,3,3,4)$ Gorenstein sequences.  

\begin{proposition}
\label{Prop:Q-decOf1334}
Let $h$ be a $(1,3,3,4)$ complete intersection sequence and  {$I$ is a
    Gorenstein ideal with $\HF_{R/I} = h$.} Then
   \begin{asparaenum} 
\item \label{Prop:UniqueDecoOfCI} if $I$ is a complete intersection, then the
    symmetric decomposition for $I$ is determined by $h$ as follows.
    Let $k$ be the peak position of $h$ and let $\delta = (0, 1, 0, \ldots , 0, 1, 0)$
    be a vector with ones at position $1$ and $k$. Then the symmetric
    decomposition $\mathfrak{D}$ of $h$ is obtained as the unique decomposition of $h -
    \delta$ into symmetric vectors, with $\delta$ added to the appropriate
    symmetric vector.
\item\label{it:partb} If $\Delta(h)=2$, then the symmetric decomposition for
    $I$ is again unique, equal to $\mathfrak{D}$ above.
\item\label{it:partc} If $\Delta(h)=1$, then the symmetric decomposition for $I$ is either
        $\mathfrak{D}$ above or $\mathfrak{D}'$ where $\mathfrak{D}'$ is the following:
let $\delta'=(0,1)$ and $(\DD(0),\DD(1),\ldots,\DD(s-2))$ be the unique symmetric decomposition of $h-\delta'$.  Then
\[
\mathfrak{D}'=(\DD(0),\DD(1),\ldots,\DD(s-2)+\delta').
\]
    \end{asparaenum}
\end{proposition}
 {Again, the symmetric decompositions are realizable in appropriate
classes. The complete intersection ideal constructed in Theorem \ref{1334} has
symmetric decomposition $\mathfrak{D}$. The decomposition $\mathfrak{D}'$ is
realized  by $A_{F'+Z^2}$ where $F' \in K_{DP}[X,Y]$ is such that
$A_{F'}$ has the Hilbert function $h-\delta'$.
}
\begin{proof} $S:=K[\![x,y]\!]$.
    Proof of~\eqref{Prop:UniqueDecoOfCI}: let $I$ be a complete intersection ideal with the Hilbert function $h.$ 
As in the proof of Proposition \ref{Prop:classificationOfArtinianGorenstein}
we may assume that the ideal $I$ contains elements of the form $f=xz-U$, $g=yz-V$,
    $p=z^2-W$,  where $U, V, W\in S_{\geq 3}$.  Let $T\in K_{DP}[X, Y, Z]$ be such that
$I=\ann_R(T).$ We write
    \begin{equation}\label{eq:TformQdecom}
        T = T_0 + Z\cdot T_1 + Z^2 \cdot T_2 +  \ldots + Z^n\cdot T_n +
        \ldots ,
    \end{equation}
    where $T_i\in
    K_{DP}[X, Y].$  Recall from the proof of Proposition
    \ref{Prop:classificationOfArtinianGorenstein} that $F:=T_0$ and $G:=T_1$
    satisfy \eqref{it:fstG} and \eqref{it:sndG} of Proposition
    \ref{Prop:classificationOfArtinianGorenstein},  and $S/{\ann_S(F,G)}$ has
    Hilbert function $h-\delta'$.
Since $\HF_{R/I}(1) = 3$, we have $Z\in \langle T\rangle$.
      Let
        $\sigma\in R$ be an element such that $\sigma\circ T = Z$.
        Subtracting an element of $(f, g, p)$ from $\sigma$, we put it in the
        form $\sigma = \sigma_0 + \lambda z$, where $\sigma_0\in S$ and
        $\lambda\in K$. Suppose $\lambda\neq 0$. Since $\sigma\circ T\in
        K_{DP}[Z]$, we have $(x\sigma)\circ T = 0 = (y\sigma)\circ T$.
        Therefore, the elements $x\sigma$, $y\sigma$ lie in
    the complete intersection $I$.  Since the quadratic parts of $x \sigma$ and $y \sigma$ are linearly independent,  they are part of a minimal generating set of $I$.  In particular,  $x \sigma, ~y\sigma$ is an $R$-regular sequence.  But the linear syzygy
    $y\cdot(x\sigma) - x\cdot (y\sigma) = 0$ has coefficients outside $I$, a
    contradiction.  This shows that $\lambda = 0$ and $\sigma\in S$.
    Expanding $T$ as in~\eqref{eq:TformQdecom}, we get
    \[
        Z = \sigma\circ F + Z(\sigma\circ G) + Z^2(\sigma \circ T_2)+\cdots+\cdots,
    \]
    which shows that $\sigma \circ G = 1$ and $\sigma\circ F = 0$. In
particular, we get $G\not\in \langle F\rangle$.

    Set $A=R/I$.  Let $\m_A\subset A$ be the maximal ideal of $A$.  Observe
    that $\sigma\mod I$ lies in $(0:\m_A^2)$. By the above considerations, the element
        $\sigma$ is unique modulo $I$.  Let $i$ be the maximal number such that
        $\sigma\mod I\in \m_A^i$.  Then $\sigma\mod I$ is a nonzero
    element of $ Q(a)_{i}$, where $a = s-1-i$,  and hence  {there} is a nonzero element of $Q(a)_1$ as well. 

       We claim that $i$ is the peak position of $h.$ First we show that  $\HF_{S/\ann_S(F, G)}$ and $\HF_{S/\ann_S(F)}$ differ
       at the position $i$.  {Denote by $\m_S$ the maximal ideal of $S$}.
        The kernel of the natural surjection
        \[
        \frac{\m_S^i}{\m_S^{i+1}+\m_S^i \cap \ann_S(F,G)} \to \frac{\m_S^i}{\m_S^{i+1}+\m_S^i \cap \ann_S(F)}
        \]
        is spanned by the class of $\sigma$,
        so the Hilbert functions of $S/\ann_S(F, G)$ and $S/\ann_S(F)$ differ
    at the position $i$. Therefore, by Lemma \ref{lemma:NumCriForGorSeqNew},
    the position $i$ is the peak position of $\HF_{S/\ann_S(F,G)}$.  Since the Hilbert function of $S/\ann_S(F,G)$ is $h$ except at position one,  $i$ is the peak position of $h.$
 Therefore   $h-\delta\geq  0$ position-wise. 
        In fact,  $h-\delta$ is a codimension two O-sequence with
        $\Delta(h-\delta)=1,$ and so  $h-\delta$ is a Gorenstein sequence by
        \eqref{Eqn:codimension2Gsequence}.  Hence by \cite[Theorem
        2.2]{Iar94}, the vector $h-\delta$ admits a unique symmetric decomposition,   say $(\DD(0),\DD(1),\ldots,\DD(s-2)).$  Then 
        \[
        \mathfrak{D}=(\DD(0),\DD(1),\ldots,\DD(a-1),\DD(a)+\delta,\DD(a+1),\ldots,\DD(s-2))
        \]
        is the unique symmetric decomposition of $R/I.$
        
        Proof of~\eqref{it:partb} and~\eqref{it:partc}: let $I$ be a
        Gorenstein ideal with the Hilbert function $h.$ Let $f,g,p,$
        $T,~\sigma$ be as in the proof of  \eqref{Prop:UniqueDecoOfCI}.  After
        adding suitable multiples of  $f,g,p$, we may assume that
        $\sigma=\sigma_0+\lambda z$, where $\sigma_0 \in K[\![x,y]\!]$ and $\lambda \in K.$ If $\lambda =0, $ then the same argument as in   \eqref{Prop:UniqueDecoOfCI}  shows that $G \notin \langle F\rangle, $ and $Q(a)_1 \neq 0$ where $a=s-1-k$ and $k$ is the peak position of $h.$ Hence $R/I$ has decomposition $\mathfrak{D}.$
        
        Suppose $\lambda\neq 0.$ Let $A=R/I$ and $\m_A$ denote the maximal
        ideal of $A.$ In this case, $\sigma\mod I \in \m_A \setminus \m_A^2.$
        Also,  $\sigma \mod I \in (0:\m_A^{2}).$  Hence $\sigma\mod I$ is a nonzero element of $Q(s-2)_1.$

We claim that in this case $\Delta(h)=1.$ Recall from the proof of Proposition \ref{Prop:classificationOfArtinianGorenstein} that $F:=T_0$ and $G:=T_1$ are the required elements in $K_{DP}[X,Y]$ in  Proposition \ref{Prop:classificationOfArtinianGorenstein}.   Since $\sigma \circ T=Z$,  using \eqref{eq:TformQdecom} we get
\[
\sigma_0 \circ F + \lambda G + Z(\sigma_0 \circ G+ \lambda T_2) + \cdots =Z.   
\]
Therefore, $\sigma_0 \circ F + \lambda G=0$ and hence $G \in \langle F\rangle.$ Thus $h_{I_F}:=\HF_{S/\ann_S(F)}$ and $h_J:=\HF_{S/\ann_S(F,G)}$ are the same.  Since $\Delta(h_{I_F})=1$ by  \eqref{Eqn:codimension2Gsequence},  $\Delta(h_J)=1.$ As $h$ and $h_J$ are the same except at position one,  $\Delta(h)=1.$ 

Therefore in this case we have the following symmetric decomposition.  Let $\delta'=(0,1).$ The sequence $h-\delta' \geq 0$ point-wise.  In fact,  $h-\delta'$ is a codimension two O-sequence with  $\Delta(h-\delta')=1.$ Hence $h-\delta'$ is a Gorenstein sequence by \eqref{Eqn:codimension2Gsequence}.  Therefore by Theorem \cite[Theorem 2.2]{Iar94} $h-\delta'$ admits a unique symmetric decomposition, say $(\DD(0),\DD(1),\ldots,\DD(s-2)).$ Then 
$$(\DD(0),\DD(1),\ldots,\DD(s-2)+\delta')$$ is the symmetric decomposition of $h.$

It is clear that if $F'\in K_{DP}[X,Y]$ and $A_{F'}$ has the Hilbert function $h-\delta',$ then $A_F$ has the decomposition $\mathfrak{D}'$ where $F=F'+Z^2.$
 \end{proof}

We illustrate the above result by the following example.

\begin{example}
 \label{Example:IarQue2}
  Let $h=(1,3,3,4,3,2,1).$  Then $h$ is admissible for a complete intersection by Theorem \ref{Equi:GorAndCIIntro}.  In fact, let 
  $F=X^4Y^2$ and $G=Y^3.$ Then $F$ and $G$ satisfy the conditions of Proposition \ref{Prop:PropertiesOfJ}.  Also, $J=\ann_S(F,G)=(x^5,y^4,xy^3)$ has  
the syzygy matrix
  \[
 \begin{bmatrix}
         y^3 & 0 & -x^4\\
         0 & -x & y
     \end{bmatrix}^{t}.
 \]
We take $U=0,V=-x^4,W=-y^3$. Then
  $I=(xz,yz+x^4,z^2+y^3)$ is a complete intersection ideal with the Hilbert function $h.$ 
  By \cite[p.94]{Iar94} the function $h$ admits two symmetric decompositions.  In fact,  each symmetric decomposition is realizable.  
  By Proposition \ref{Prop:Q-decOf1334}\eqref{Prop:UniqueDecoOfCI}, Hilbert
  function of $R/I$ has the decomposition
  \[
 \DD= (\DD(0)=(1,2,3,3,3,2,1),\DD(2)=(0,1,0,1)).
  \]
The other symmetric decomposition of $h$,  namely
 \[
\mathfrak{D}'= (\DD(0)=(1,2,3,4,3,2,1),\DD(4)=(0,1)),
 \]
  is realized by $A_F$ where $F=X^3Y^3+Z^2.$ By Proposition
  \ref{Prop:Q-decOf1334} the decomposition $\mathfrak{D}'$ is not realizable by a complete intersection ideal.
  
 \end{example}
 
\begin{proof}[Proof of Theorem \ref{thm:Q-deco-133}] This follows from Propositions \ref{Prop:Q-decOf133} and \ref{Prop:Q-decOf1334}.
\end{proof}

\section*{Acknowledgements}
We thank Prof.~Tony Iarrobino for detailed comments on an earlier version of the draft. 
In particular,  we thank him for suggesting the problem on realizable symmetric decomposition 
of $(1,3,3)$ sequences. We also thank the anonymous referee, whose numerous
helpful comments improved the presentation considerably.

\newcommand{\etalchar}[1]{$^{#1}$}


\begin{thebibliography}{}

\bibitem[BEH{\etalchar{+}}21]{On_the_infinite_loop_spaces}
T.  Bachmann,  E.  Elmanto,  M.  Hoyois,  A.~A. Khan,  V.  Sosnilo,  and
  M.  Yakerson,  {\em On the infinite loop spaces of algebraic cobordism and the motivic  sphere},  
\newblock {\'{E}pijournal G\'{e}om. Alg\'{e}brique}, 5:Art. 9, 13, 2021.



\bibitem[Ber09]{Ber09} V.~Bertella, {\em Hilbert function of local Artinian level rings in codimension two}, J. Algebra {\bf 321} (2009), 1429-1442.

\bibitem[Bri77]{Bri77} J.~Brian\c{c}on, {\em Description de $\Hilb^n \mathbb{C}\{x,y\}$}, Invent. Math. {\bf 41} (1977), no. 1, 45-89 

\bibitem[BH98] {BH98} W. Bruns and J. Herzog, {\em Cohen-Macaulay rings}, Revised Edition, 
Cambridge University Press, 1998.

\bibitem[BE77]{BE77} D.~A.~Buchsbaum and D.~Eisenbud, {\em Algebra structures for finite free resolutions, and some structure theorems for ideals of codimension 3}, Amer. J. Math. {\bf 99} (1977), no. 3, 447-485.

\bibitem[CJN15]{cjn13}
G.  Casnati,  J. ~Jelisiejew and R. Notari,  
{\em Irreducibility of the {G}orenstein loci of {H}ilbert schemes via ray
  families},   Algebra Number Theory,  {\bf 9} (2015),  no.  7,  1525-1570.

 \bibitem[Singular] {Singular} W.~Decker, G.-M.~Greuel, G.~Pfister and H.~Sch{\"o}nemann, \newblock {\sc Singular} {4-2-1} --- {A} computer algebra system for polynomial computations available at 
\newblock {http://www.singular.uni-kl.de},  (2021).


\bibitem[Eli15]{Elias} J.~Elias, \emph{{\sc Inverse-syst.lib}--{S}ingular library for computing
  {M}acaulay's inverse systems}, http://www.ub.edu/C3A/elias/inverse-syst-v.5.2.lib, 2015.

\bibitem[Ems78] {Ems78}
J.~Emsalem, {\em G\'eom\'etrie des points \'epais}, Bull. Soc. Math. France
  \textbf{106} (1978), no.~4, 399-416.
  
 
 \bibitem[GHK06]{GHK06}
S.~Goto, W.~Heinzer, and M.-K. Kim, \emph{The leading ideal of a complete
  intersection of height two}, J. Algebra \textbf{298} (2006), no.~1, 238-247.


  \bibitem[GHK07]{GHK07}
S.~Goto, W.~Heinzer, and M.-K. Kim, \emph{The leading ideal of a complete
  intersection of height two, Part II }, J. Algebra \textbf{312} (2007), no.~2, 709-732.


\bibitem[GP08]{GP08}  { G.-M. Greuel and G. Pfister, {\em A Singular Introduction to Commutative Algebra}, second, extended edition, Springer, Berlin, ISBN 978-3-540-73541-0, 2008, xx+689 pp., with contributions by Olaf Bachmann, Christoph Lossen and Hans Schönemann; with CD-ROM (Windows, Macintosh and UNIX).}

\bibitem[Macaulay2]{Macaulay2} D.~Grayson and M.~Stillman, {\em Macaulay2, a software system for research in algebraic geometry}, available at \url{http://www.math.uiuc.edu/Macaulay2/}.

\bibitem[HW20]{HW20} R.~Homs and A.L~Winz,  {\em Canonical Hilbert-Burch matrices  for power series},  J.  Algebra  {\bf 583} (2021),  1-24.

\bibitem[Iar77]{Iar77} A.~Iarrobino,  {\em Punctual Hilbert schemes},  Mem. Amer. Math. Soc.  {\bf 10} (1977),  no. ~188,  viii+112 pp.

\bibitem[Iar83]{Iar83} A.~Iarrobino,  {\em Deforming complete intersection Artin algebras. Appendix: Hilbert function of $\mathbb{C}[x,y]/I$},  Singularities, Part 1 (Arcata, Calif., 1981),  593-608,
Proc.  Sympos. Pure Math.,  {\bf 40},  Amer. Math. Soc.,  Providence,  R.I.,  1983.

\bibitem[Iar84]{Iar84} {A.~Iarrobino,  {\em Compressed algebras: {A}rtin algebras having given socle
 degrees and maximal length}, Trans. Amer. Math. Soc. {\bf 285} (1984),  337-378.}

 \bibitem[Iar94]{Iar94}
A.~Iarrobino, {\em Associated graded algebra of a {G}orenstein {A}rtin algebra},
  Mem. Amer. Math. Soc. \textbf{107} (1994), no.~514, viii+115 pp.
  
    \bibitem[IM21A]{IM20}  {A.~Iarrobino and  P. M.  Marques, {\em Symmetric decomposition of the associated graded algebra of an Artinian Gorenstein algebra},  J. Pure Appl. Algebra {\bf 225} (2021),  no. 3, 106496,  49 pp.}
  
  \bibitem[IM21B]{IM21} A.~Iarrobino and  P. M.  Marques, {\em Reducibility of a family of local Artinian Gorenstein algebras},  (2021),  Preprint,  {available at \arxiv{2112.14664}}.
  
  \bibitem[IK99]{ik99}
A.~Iarrobino and V.~Kanev, {\em Power sums, {G}orenstein algebras, and
  determinantal loci}, Appendix C by Iarrobino and Steven L. Kleiman, Lecture Notes in Mathematics, vol. 1721, Springer-Verlag, Berlin, 1999. 

\bibitem[Kot78]{Kot78}
S.~C.  Kothari, \emph{The local {H}ilbert function of a pair of plane curves},
  Proc. Amer. Math. Soc. \textbf{72} (1978), no.~3, 439--442.
  
\bibitem[Jel17]{Jel17}  {J.~Jelisiejew, {\em Classifying local Artinian Gorenstein algebras}, Collect. Math. {\bf 68} (2017), no. 1, 101-127.}

\bibitem[Mac1904]{Mac04} F.~S.~Macaulay, {\em On a method of dealing with the intersections of plane curves}, Trans. Amer. Math. Soc. {\bf 5} (1904), no. 4, 385-410.

\bibitem[Mac1916]{Mac1916} F.~S.~Macaulay, {\em The algebraic theory of modular systems}, 
Cambridge Mathematical Library, Cambridge University Press, Cambridge, 1994. Revised reprint of the 1916 original; With an introduction by Paul Roberts.



\bibitem[MR15]{MR15}  {M.  Mandal and M.  E.  Rossi,  
{\em The tangent cone of a local ring of codimension 2},
Acta Math. Vietnam.  {\bf 40} (2015),  no. 1,  85-100.}

  \bibitem[MR18]{MR18} S. K. Masuti and M. E. Rossi, {\em Artinian level algebras of socle degree 4}, J. Algebra {\bf 507} (2018), 525-546.

    
\bibitem[Nor74]{Nor74} D. G. Northcott, 
{\em Injective envelopes and inverse polynomials},  J.  London Math. Soc. (2) {\bf 8} (1974), 290-296.   


    
\bibitem[Sta78]{stanley} R.~Stanley, {\em Hilbert functions of graded algebras}, Adv. in Math. {\bf 28} (1978), 57-83.

\end{thebibliography}
\end{document}